\documentclass[10pt]{article}
\RequirePackage[OT1]{fontenc}
\RequirePackage{amsthm,amsmath}
\RequirePackage[comma,sort&compress]{natbib}
\RequirePackage[colorlinks=true]{hyperref}
\hypersetup{
    linkcolor = {blue},
    citecolor = {blue},
    urlcolor = {blue}
    }
\usepackage[utf8]{inputenc}
\usepackage{times}
\usepackage{fullpage}
\usepackage{dsfont,amsmath, amsfonts, amsthm, amssymb}
\usepackage{bm}
\usepackage{color}
\usepackage{xspace}
\usepackage{bold-extra}
\usepackage{microtype}
\usepackage{newpxmath} 
\usepackage[inline,shortlabels]{enumitem}
\usepackage{authblk}
\usepackage{float}
\usepackage{mathtools}
\numberwithin{equation}{section}
\newtheorem{theorem}{Theorem}[section]

\newtheorem{remark}[theorem]{Remark}
\newtheorem{lemma}[theorem]{Lemma}
\newtheorem{proposition}[theorem]{Proposition}
\newtheorem{definition}[theorem]{Definition}
\newtheorem{hyp}{Hypothesis}
\newcommand{\Blue}[1]{{\color{blue}{#1}\color{black}\xspace}}

\newcommand{\av}[1]{{{\displaystyle{\langle} {#1}\displaystyle{\rangle}}}}
\newcommand{\p}{\partial}
\newcommand{\pp}{\partial^2}
\newcommand{\f}[2]{\av{\partial_{\omega} #1 , #2}}
\newcommand{\ff}[3]{\av{\partial^2_{\omega\omega} #1 , (#2, #3)}}
\newcommand{\leqs}{\lesssim}
\newcommand{\ind}{\mathds{1}}
\newcommand{\indic}[1]{\mathds{1}_{\{#1\}}}
\newcommand{\one}[1]{\mathds{1}_{#1}}
\newcommand{\func}[3]{ {#1}:{#2} \xrightarrow{} {#3}}

\newcommand{\bo}[1]{\boldsymbol{#1}}
\newcommand{\abs}[1]{\left\lvert#1\right\rvert}
\newcommand{\norm}[1]{\left\lVert#1\right\rVert}

\newcommand{\indepp}{\perp\!\!\!\!\!\perp}
\newcommand{\EE}{\mathbb{E}}
\newcommand{\PP}{\mathbb{P}}
\newcommand{\RR}{\mathbb{R}}
\newcommand{\NN}{\mathbb{N}}

\newcommand{\Ff}{\mathcal{F}}

\newcommand{\Ee}{\mathcal{E}}

\newcommand{\Xx}{\mathcal{X}}
\newcommand{\Yy}{\mathcal{Y}}
\newcommand{\ak}{\overline{K}}
\newcommand{\vol}{V}
\newcommand{\itgr}{X}
\newcommand{\bvol}{\overline{\vol}}
\newcommand{\bitgr}{\overline{\itgr}}
\newcommand{\btheta}{\overline{\Theta}}
\newcommand{\brho}{\hat{\rho}}
\newcommand{\flowx}{X^{t,x,\omega}}

\newcommand{\Flowx}{X^{t,\bitgr_t,\btheta^t}_T}
\newcommand{\Flowv}{V^{t,\btheta^t}_s}
\newcommand{\bxi}{\overline{X}}

\newcommand{\normt}[1]{\| #1 \|_{[0,T]}}
\newcommand{\mall}[2]{\mathbf{(} \mathbf{D}_{#1} #2 | \boldsymbol{\rho} \mathbf{)}}
\newcommand{\mallext}[2]{\mathbf{\bigg(} \mathbf{D}_{#1} \Big\{ #2 \Big\} \Big| \boldsymbol{\rho} \mathbf{\bigg)}}
\newcommand{\mallw}{\mathbf{D}^{W}}
\newcommand{\mallbw}{\mathbf{D}^{\widehat{W}}}
\newcommand{\var}{\textrm{$\mathbb{V}$\!ar}} 
\newcommand*{\doi}[1]{\href{http://dx.doi.org/#1}{\footnotesize{doi:#1}}}

\title{Weak rough kernel comparison via PPDEs for integrated Volterra processes}
\author[1]{Mireille Bossy\footnote{mireille.bossy@inria.fr}}
\author[2]{Kerlyns Mart\'inez\footnote{kermartinez@udec.cl}}
\author[1]{Paul Maurer\footnote{paul.maurer@inria.fr}}
\affil[1]{Universit{\'e} C{\^o}te d'Azur, Inria, France}
\affil[2]{Departamento de Ingenier\'ia Matem\'atica, Universidad de Concepci\'on, Chile}
\date{January 2025}
\providecommand{\keywords}[1]{\textbf{\textit{Keywords---}} #1}
\begin{document}
\maketitle

\begin{abstract}
Motivated by applications in physics (e.g., turbulence intermittency) and financial mathematics (e.g., rough volatility), this paper examines a family of integrated stochastic Volterra processes characterized by a small Hurst parameter $H<\tfrac{1}{2}$. We investigate the impact of kernel approximation on the integrated process by examining the resulting weak error. Our findings quantify this error in terms of the $L^1$ norm of the difference between the two kernels, as well as the $L^1$ norm of the difference of the squares of these kernels. Our analysis is based on  a path-dependent Feynman-Kac formula and the associated partial differential equation (PPDE), providing a robust and extendible framework for our analysis.
\end{abstract}

\tableofcontents

\section{Introduction}\label{sec:introduction}

For a fixed final time horizon $T>0$,
we  consider $\func{b,\sigma}{[0,T] \times \RR}{\RR}$ and $\func{K}{[0,T]}{\RR}$, some Borel-measurable functions, where the kernel $K$ may be singular at the origin (see the precise definition of this singularity in Hypothesis \ref{hyp:kernel} in Section \ref{sec:main}). In this setting, we consider the two-dimensional  stochastic integrated Volterra process $(X,V)$ given by
\begin{equation} \label{eq:integrated_model}
\begin{aligned}
\itgr_t = & X_0 + \int_0^t b(s,\vol_s) ds + \int_0^t \sigma(s,\vol_s) \ d B_s, \\
\vol_t = & \int_0^t K(t-u) \ dW_u, \quad t \in [0,T],
\end{aligned}
\end{equation}
where for  $\rho \in [0,1]$, $\brho = \sqrt{1-\rho^2}$,  $(W,B)  = (W,\rho W + \brho \ \widehat{W})$ is a correlated two-dimensional Brownian motion, issued from the standard two-dimensional one $\bo{W} = (W,\widehat{W})$ that equips the  filtered probability space  $(\Omega,\Ff,(\Ff_t)_{t \geq 0},\PP)$.
\medskip

In many applications, observation data can exhibit long-term or short-term dependence, characterized by persistent correlations, as well as self-similarity, when the process displays similar behaviour across different time scales. These properties are observed in financial data and can be effectively modelled using integrated Volterra processes \eqref{eq:integrated_model}. For instance, the so-called rough Bergomi model, introduced in \cite{Gatheral03062018, bayer2016pricing, fukasawa2019volatility} within the context of financial mathematics, can be interpreted as an integrated Volterra process by setting 
\[K(r) \propto r^{H-\frac{1}{2}}, \quad b(t,x) = - \tfrac{1}{2} \exp\left(\nu x - \frac{\nu^2}{2} t^{2H}\right)\quad\text{and}  \quad \sigma(t,x) = \exp\left( \frac{\nu x}{2} - \frac{\nu^2}{4} t^{2H}\right),\]
for a fixed $\nu >0$. This model prescribes that the $\log$-price $X$ of an asset satisfies
\begin{equation*}
\begin{aligned}
X_t & =   \int_0^t \exp\left(\frac{\nu}{2}  C_H^{-1} V_s - \frac{\nu^2}{4} t^{2H}\right) dB_s - \frac{1}{2} \int_0^t \exp\left(\nu  \  C_H^{-1} V_s - \frac{\nu^2}{2} t^{2H}\right)  ds,\\
V_t & = C_H \int_0^t \big(t-u\big)^{H-\frac{1}{2}} d W_u,
\end{aligned}
\end{equation*}
where  the rough volatility $V$  is parametrized with $H \in (0,\frac{1}{2})$ and a constant $C_H > 0$, commonly considered as $C_H = \sqrt{2H}$.
\medskip

In fluid mechanics,  Kolmogorov's refined theory provides a statistical framework that predicts multiscaling behaviour, such as anomalous power laws. These phenomena emerge at the level of fluid velocity increments, which are associated with fluctuations in a related quantity known as energy dissipation, denoted by $\varepsilon_t$ (\cite{kolmogorov1962refinement, frisch1995turbulence}). This behaviour is often referred to as an intermittency phenomenon.

Formulated in the language of multifractal stochastic processes, the Kolmogorov's refined theory describes the behaviour of the dissipation process as an exponentially correlated random process with particular multiscale behaviour for both one-time statistics of ($\varepsilon_t$) and integrated statistics of ($\int_0^t\varepsilon_s ds$). Then, the dissipation process can be approximate as the limit when $H$ goes to zero of integrated Volterra process $(X, V)$ (see \cite{FORDE2022109265}), satisfying 
\begin{equation}\label{eq:dissip_process}
\begin{aligned}
X_t = \int_0^t\varepsilon^H_s ds, \quad \varepsilon^H_t & \propto \exp(\nu \ V_t - \frac{\nu^2}{2} \var[V_t]), \qquad
V_t =  \int_0^t (t-u)^{H-\frac{1}{2}} d W_u,
\end{aligned}
\end{equation}
where $\nu$ is known as the intermittency parameter. 

\subsection*{Motivation for the approximation problem}  
To approximate in time and simulate the integrated Volterra process \eqref{eq:integrated_model}, one can employ the Euler-Maruyama scheme.  
The weak convergence rate of the Euler-Maruyama scheme in the context of rough volatility models has been analysed in \cite{bayer2020weak}, \cite{gassiat2023weak}, and \cite{bonesini2023rough}. However, implementing the Euler scheme to simulate the integrated process $X$ requires  increment sampling of the underlying Volterra process $V$. 
Exact sampling is possible, but its computation may become prohibitively  demanding.  For instance, when considering the dissipation process in \eqref{eq:dissip_process}, the limit kernel in $H\rightarrow 0$ introduces significant challenges for both approximation and simulation due to the explosion of variance. 
This singularity dramatically increases the computational cost of accurately modelling such processes. As a result, it is crucial to quantify some efficient kernel approximations that balance computational feasibility with the need to capture the essential characteristics of the process. A natural approach is first to approximate the kernel $K$.

A typical example of such kernel approximation techniques is the Markovian approximation approach, which goes back to \cite{carmona2000approximation}. This techniques  proposes to approximate the kernel $K$, thought its Laplace transform  by a sum of time-exponential Brownian integrals leading to the simulation of a set of correlated Ornstein-Uhlenbeck processes. The general idea can be outlined as  follows. Consider a positive square-integrable function $K$ on $[0,T]$ that  is completely monotone, meaning that there exists a positive non-decreasing function $\lambda$ such that 
\begin{equation*}
K(t) = \int_0^{+\infty} e^{-t x} \lambda(x) dx, \quad t \in [0,T].
\end{equation*}
Applying the stochastic Fubini theorem (see e.g. \cite{revuz2013continuous}), we can write the Volterra process as
\begin{align*}
V_t =  \int_0^t K(t-u) dW_u = \int_0^t \left( \int_0^{+\infty} e^{-(t-u) x} \lambda(x) dx \right) dW_u = \int_0^{+\infty} \left( \int_0^t e^{-(t-u) x} dW_u \right) \lambda(x) dx.
\end{align*}
For a positive integer $n$, let $(x_i,w_i)_{ \{ 1 \leq i \leq n \} }$ describes a quadrature rule for integrals of the form $\int_0^{+\infty} f(x) \lambda(x) dx$, where $f$ is a continuous function. Using this rule, the Volterra process can be approximated by
\begin{equation*}
\bvol_t = \sum_{i=1}^n w_i Y_t^{(i)},
\end{equation*}
where the processes $Y^{(i)} = \int_0^\cdot e^{-(\cdot-u) x_i} dW_u$ are correlated Ornstein-Uhlenbeck processes starting from zero:
\begin{equation*}
  d Y_t^{(i)} = - x_i Y_t^{(i)} dt + dW_t, \quad Y_0^{(i)} = 0.
\end{equation*}
In this way, the processes $Y^{(i)}$, and consequently the process $\bvol$, can be simulated efficiently, either by exact simulation or by using the Euler-Maruyama scheme.
Additionally, the process $\bvol$ can be expressed as 
\begin{equation*}
  \bvol_t = \int_0^t \ak(t-u) dW_u, \quad \text{with} \quad \ak(t) = \sum_{i=1}^n w_i e^{-t x_i}.
\end{equation*}
Therefore the Markovian approximation of $V$ can be interpreted as approximating $K$ by the above kernel $\ak$.\medskip

Strong  convergence rates for the Markovian approximation have been established in   \cite{carmona2000approximation}, \cite{harms2021} and \cite{alfonsi2024}, while weak convergence rates have been studied more recently, for example, in \cite{bayer2023markovian} and \cite{bayer2023weak}.\medskip

When dealing with singular kernel, Markovian approximation may be combined with a kernel smoothing. 
 For the fractional kernel $K(t) = t^{H-\frac{1}{2}}$, with small $H$, two common approximation methods are: the smoothed fractional kernel  
\begin{equation}\label{eq:kernel_appro_smoothed}
\ak_1(t) = (t+\tau)^{H-\frac{1}{2}}, 
\end{equation}
or the truncated approximation
\begin{equation}\label{eq:kernel_appro_truncated}
\ak_2(t) = t^{H-\frac{1}{2}} \indic{t > \tau} + \tau^{H-\frac{1}{2}} \indic{t \leq \tau},
\end{equation}
where $\tau > 0$ is a  smoothing/truncation parameter. 

Obviously, kernel approximation errors affect the accuracy of the  process $X$. This concern motivates our study of errors arising from the approximation of  $X$, particularly in the context of  developing a semimartingale approximation of the process $V $.

\subsection*{Main results}
Given the positive kernel $K \in L^2([0,T],\RR_+)$ and the corresponding integrated Volterra process $(X,V)$ in \eqref{eq:integrated_model}, we  consider another  kernel $\ak \in L^2([0,T],\RR_+)$, and the related modified integrated Volterra process $(\bitgr, \bvol)$ of $(X,V)$,  given by
\begin{align*}
\bitgr_t = & X_0 + \int_0^t b(s,\bvol_s) ds + \int_0^t  \sigma(s,\bvol_s) d B_s \\
\bvol_t = & \int_0^t \ak(t-u) dW_u.
\end{align*}
Our main result stated in Theorem \ref{thm:weakcvg_integrated} quantifies the weak error between the integrated components  $X$ and $\bitgr$ through the $L^1$-norm of $K-\ak$ and $K^2 - \ak^2$, namely 
\begin{equation*}
| \EE[\phi(\itgr_T)] - \EE[\phi(\bitgr_T)] | \leq C_K \int_0^T \Big\{ \|{K - \ak}\|_{L^1([0,t])} + \|{K^2 - (\ak)^2}\|_{L^1([0,t])} \Big\} dt. 
\end{equation*}
This bound is obtained under assumptions of regularity for the test function $\phi$ and the coefficients $b$ and $\sigma$.  The kernel $K$ is assumed to be of fractional type with a Hurst parameter $H\in(0,\tfrac{1}{2})$.  The constant $C_K$ is explicit in terms of the kernels' $L^2([0,T])$-norm.
\medskip

\subsection*{Method of proof} 
\paragraph{Path-dependent Feynman-Kac formula. } The standard  method for obtaining a weak convergence rate in the Markovian case (consider, for instance, the weak error for the  Euler-Maruyama scheme) relies on the Kolmogorov forward equation and Itô's formula. 
In our case, these two tools are inaccessible because the  process $V$ is non-Markovian and, in general, is not a semimartingale when $K$ is singular. However, recent works by \cite{viens2019} have uncovered a Markovian structure within Volterra processes. The trade-off is that the state space then becomes infinite-dimensional. 
The general idea relies on the following orthogonal decomposition: for $s \geq t \in [0,T]$, we write
\begin{align*}
 V_s = \int_0^{t} K(s-r) dW_r + \int_{t}^{s} K(s-r) dW_r =  \Theta^{t}_{s} + I^{t}_{s}, 
\end{align*}
where we have $\Theta^{t}_{s} \in \ \Ff_{t}$  and $I^{t}_{s} \ \indepp \ \Ff_{t}$. 
From this observation, one deduces that for any test function $\func{\phi}{\RR}{\RR}$,
\begin{align*}
\EE[\phi(X_T) | \Ff_t] & =  \EE\left[ \phi\left( X_{t}  + \int_{t}^T b(V_{s}) d{s} \right) \big| \Ff_{t} \right]  
= \EE\left[ \phi\left( X_{t} + \int_{t}^T  b(\Theta^{t}_{s} + I^{t}_{s}) d{s} \right) \big| \Ff_{t} \right]
  \\ 
& = u(t,X_{t},\Theta^{t}_{[t,T]}),
\end{align*}
with $u(t,x,\omega) = \EE [\phi(X_T^{t,x,\omega})] =  \EE[\phi(X_T) | X_t = x, \Theta^t_{[t,T]} = \omega]$ for all $(t,x,\omega) \in [0,T] \times \mathbb{R} \times C([t,T],\mathbb{R})$.

Consequently, the information contained in the conditional expectation of $\phi(X_T)$ given the Brownian filtration $\Ff_t$ is entirely described by the information contained in $X_t$ and the collection $(\Theta^t_s)_{s \in [t,T]}$. Provided that a satisfactory notion of derivatives for functions with variables in $C^0([0,T])$ can be defined, it then seems reasonable to hope for an equivalent formula to the Kolmogorov backward equation for such a function $u$, thereby enabling a fairly standard analysis of the weak error. Viens and Zhang thus state their trajectory-dependent Itô formula using the Fréchet derivatives of $u$ with respect to $\omega$, under sufficient regularity conditions to include cases where the kernel $K$ is singular. The trajectory-dependent Feynman-Kac formula, suggested in Viens and Zhang's paper, is rigorously demonstrated in \cite{bonesini2023rough}.

\paragraph{Directional derivatives. }
However, the notion of Fréchet derivative tends to be restrictive in this context, as it does not allow for a satisfactory chain rule. Indeed, if $\func{\phi}{\RR}{\RR}$ is a $C^2$ function with $|\phi(x)| + |\phi'(x)| + |\phi''(x)| \leqs 1 + |x|^{\kappa}$, it is unfortunately not true in general that its extension $\tilde{\phi} : L^p(\Omega) \to L^1(\Omega)$, defined by $\tilde{\phi}(X)(\omega) = \phi(X(\omega))$, is Fréchet-differentiable between the spaces $(L^p(\Omega), \|\cdot\|_{L^p(\Omega)})$ and $(L^1(\Omega), \|\cdot\|_{L^1(\Omega)})$. This requires proving the Fréchet-differentiability of all derivatives of $u$ by hand. To streamline these calculations, we instead propose using the notion of Hadamard derivative, which is the weakest notion of derivative that preserves a chain rule (see, for instance, \cite{shapiro1990}).

Indeed, it is possible to show under suitable conditions on $p$ that the above function $\tilde{\phi}$ is Hadamard-differentiable (see Lemma \ref{lm:hadamard_differentiability} in the Appendix below). Then, it suffices to establish the Hadamard-differentiability of $\omega \mapsto X_T^{t,x,\omega}$ in $L^p$ to compute all the derivatives of $\omega \mapsto \EE[\phi(X_T^{t,x,\omega})]$ via the chain rule. Furthermore, a careful reading of the proof of the Itô formula in \cite{viens2019} reveals that the Fréchet nature of the derivatives is not necessary. It suffices to have a Gâteaux derivative that is linear and continuous with respect to its direction. Continuity is already ensured by Hadamard-differentiability, while linearity may usually be directly observed in the obtained derivatives.

For these reasons, we believe that the use of Hadamard derivatives could simplify calculations in future work using these new techniques for more complex cases of functional Itô formulas and path-dependent PDEs. 

\subsection*{Plan of the paper}
In Section \ref{sec:volterra_processes}, as an introduction to the weak estimation proof techniques, we derive the similar estimation on $|\EE[\phi(\bvol_T)] - \EE[\phi(\vol_T)]|$  (see Lemma \ref{lem:volterra_processes}).  
The precise hypotheses and the main theorem are stated in Section \ref{sec:main}. Section \ref{sec:PPDES} is devoted to the proof of the path-dependent  Feynman-Kac formula and the strong regularity of the solution of the Kolmogorov path-dependent PDE. 
Probabilistic expressions of the derivatives of the  Feynman-Kac formula are also developed, which play an essential role in the proof of weak estimates. The proof of the weak error bound is detailed in Section \ref{sec:proof_thm_1}, complemented with Appendix \ref{sec:appendix}.

\subsection*{Notations}
~ \xspace ~ ~ $\bullet$ We let $\NN$ be the set of non-negative integers and $\RR_+$ the set of non-negative real numbers. 
If $(X,\Xx)$ and $(Y,\Yy)$ are two topological spaces, we note $C^0(X,Y)$ (resp. $D^0(X,Y)$) the space of continuous functions (resp. càdlàg functions) from $X$ to $Y$.
If $(X,\| \cdot \|_X)$ and $(Y,\| \cdot \|_Y)$ are two Banach spaces and $p \geq 1$, we note $L^p(X,Y)$ the space of functions $f$ from $X$ to $Y$ such that $\| f \|_{L^p(X,Y)} = \left( \int_X \| f(x) \|_Y^p dx \right)^{\frac{1}{p}} < + \infty$. 
We note $L^p(X)$ the space $L^p(X,\RR)$ and we note $\| \cdot \|_{L^p(X)}$ for $\| \cdot \|_{L^p(X,\RR)}$. We say that a function $\func{f}{X}{Y}$ is Lipschitz if there exists $C>0$ such that $\| f(x) - f(y) \|_Y \leq C \| x - y \|_X$ for every $x,y \in X$, in which case we note $\| f \|_{\text{Lip}} \coloneq \inf_{\| x - y \|_X \leq 1, x \neq y} \{\frac{\| f(x) - f(y) \|_Y}{\| x - y \|_X}\} $ its Lipschitz constant.

$\bullet$ If $x$ and $y$ are two real number, we note $x \wedge y$ (resp. $x \vee y$) the minimum (resp. the maximum) between $x$ and $y$. 

$\bullet$ If $f$ and $g$ are two positive functions, we note $f(x) \leqs g(x)$ if there exists $C > 0$ not dependent on $x$ such that $f(x) \leq C g(x)$. In this paper, we will make use of this notation only when the constant $C$ does not depend on the kernel $K$.

$\bullet$ Let $\func{\omega}{[0,T]}{\RR}$ and $\func{\theta}{[0,T]}{\RR}$ be Borel-measurable functions. We define their concatenation path at time $t$ by
\begin{equation*}
    \forall u \in [0,T], \qquad (\omega \oplus_t \theta)_u = \begin{cases}
        \omega_u, & \text{if } u \leq t, \\
        \theta_u, & \text{if } u > t.
    \end{cases}
\end{equation*}

$\bullet$ We note $\textbf{D}$ the Malliavin derivative operator with respect to $\bo{W} = (W,\widehat{W})$ and $\mathbb{D}^{1,2}(W,\widehat{W})$ its domain of application in $L^2(\Omega)$. For $F \in \mathbb{D}^{1,2}(W,\widehat{W})$ and $s \in [0,T]$, we note $\mall{s}{F} \coloneq \rho \mallw_s F + \brho \mallbw_s F$ with $\mallw$ and $\mallbw$ the Malliavin derivatives with respect to $W$ and $\widehat{W}$.

$\bullet$ In what follows, we will consider a function $K \in L^2([0,T])$, called kernel, and a function $K \in L^2([0,T])$, called approximated kernel. We will extensively make use of the following notations in the proofs:
\begin{equation}
\begin{aligned}
\Delta K \coloneq \ak - K, \qquad \Delta (K^2) \coloneq \ak^2 - K^2,\\
\Sigma K \coloneq \ak + K, \qquad \Sigma (K^2) \coloneq \ak^2 + K^2.
\end{aligned}
\end{equation}
  
\section{Preamble on Volterra processes} \label{sec:volterra_processes}

As a first attempt at quantifying the approximation error $|\EE[\phi(X_T)] - \EE[\phi(\bitgr_T)]|$, in this section, we derive an estimation of the  weak error $|\EE[\phi(\vol_T)] - \EE[\phi(\bvol_T)]|$ in order to introduce the approach we will then use to deal with the integrated process. 

Since the random variable $V_T$ and $\bvol_T$ are Gaussian, a first estimation can be obtained by Gaussian techniques. In particular, we have the following result for the Wasserstein-$1$ distance between two centred Gaussian laws 
(see \citet[Theorem III.1]{chhachhi20231wasserstein}): 
\begin{equation*} 
\mathcal{W}_1\left( \mathcal{N}(0,\sigma^2) , \  \mathcal{N}(0,\overline{\sigma}^2)\right) 
= (1-{\tfrac {2}{\pi }}) \   | \sigma - \overline{\sigma} |.
\end{equation*}

Applying the latter result with $\sigma^2 = \int_0^T K(T-t)^2 dt$ and $\overline{\sigma}^2 = \int_0^T \ak(T-t)^2 dt$,  we obtain that for Lipschitz functions $\phi$ satisfying $| \phi |_{\text{Lip}} \leq 1$,
\begin{equation}\label{eq:gaussian_wasserstein_distance}
\begin{aligned}
|\EE[\phi(\bvol_T)] - \EE[\phi(\vol_T)]| 
& \leq 
(1-{\tfrac {2}{\pi }}) \left|\| K\|_{L^2([0,T])} - \|\ak\|_{L^2([0,T])} \right| \\ 
& \quad = \frac{(1-{\tfrac {2}{\pi }})}{(\| K\|_{L^2([0,T])} + \| \ak\|_{L^2([0,T])})} \left| \int_0^T (\ak(t)^2 - K(t)^2) dt \right|. 
\end{aligned}
\end{equation}
Such Gaussian estimate cannot be extended to the integrated processes $(X,\bitgr)$. However, by introducing a PDE approach, we obtain the following  slightly lower estimate, disregarding any assumptions on $\phi$. 
\begin{lemma}\label{lem:volterra_processes}
For a test function $\phi \in C^2(\RR)$ such that, for all $x\in \RR$,  $|\phi(x)| + |\phi'(x)| + |\phi''(x)| \leqs (1+|x|^{\kappa_{\phi}})$, for some   $\kappa_{\phi}>0$, 
\begin{equation}\label{eq:weak_error_volterra_process}
|\EE[\phi(\bvol_T)] - \EE[\phi(\vol_T)]| \leq C_{T,K,\ak,\phi}  \|\ak^2 - K^2\|_{L^1([0,T])}.
\end{equation}
\end{lemma}
This approach then can be extended to $(X, \bitgr)$ as we will demonstrate in the rest of the paper. 
Note that, in the application case of the Markovian approximation introduced in Section \ref{sec:introduction}, choosing a Gaussian quadrature will systematically result in $\ak \leq K$ (see  e.g., \citet[Corollary D.2]{bayer2023weak}), and consequently, the two estimates above will coincide on the order $\|\ak^2 - K^2\|_{L^1([0,T])}$. 
\medskip

A similar result is provided in the work of \cite{alfonsi2024}, where the authors establish an $L^2$-strong kernel-comparison theorem for the case of fully stochastic Volterra equations.
More precisely, defining $(V, \ak)$ satisfying 
\begin{equation}\label{eq:volterra_sdes}
\begin{aligned}
V_t & = v_0 + \int_0^t K(t-s) b(V_s) ds + \int_0^t K(t-s) \sigma(V_s) dW_s, \\
\overline{V}_t & = v_0 + \int_0^t \ak(t-s) b(\overline{V}_s) ds + \int_0^t \ak(t-s) \sigma(\overline{V}_s) dW_s,
\end{aligned}
\end{equation}
with $v_0 \in \RR$ and $b,\sigma$ being globally Lipschitz functions, the authors showed that there exists a constant $\tilde{C}_{T,K,\ak}$ such that for any $t \in [0,T]$,
\begin{equation*} 
\EE[ |V_t - \overline{V}_t|^2] \leq \tilde{C}_{T,K,\ak} \int_0^t |\ak(s) - K(s)|^2 ds.
\end{equation*}
Consequently, for any Lipschitz function $\phi$,
\begin{equation}\label{eq:alfonsi_strong_error}
\left|\EE[\phi(\bvol_T)] - \EE[\phi(\vol_T)]\right| \leq \| \phi \|_{\text{Lip}} \ \tilde{C}_{T,K,\ak}^{1/2} \ \|\ak - K\|_{L^2([0,T])}.
\end{equation}
It comes out that $\|\ak - K\|_{L^2([0,T])}$ in \eqref{eq:alfonsi_strong_error} is less accurate than $\|\ak^2 - K^2\|_{L^1([0,T])}$ in \eqref{eq:weak_error_volterra_process} and \eqref{eq:gaussian_wasserstein_distance}. However, the direct use of a strong probabilistic error already highlights the complexity of Volterra SDEs \eqref{eq:volterra_sdes}.\medskip

The interest of the technique employed to derive \eqref{eq:weak_error_volterra_process} lies in its adaptability to more general stochastic models, including the integrated model \eqref{eq:integrated_model}. This is made possible by recent advances in path-dependent PDEs and thanks to the functional Itô formula, as introduced in \cite{viens2019} and \cite{bonesini2023rough}. These tools even enable the derivation of an $L^1$-type error bounds in some specific cases. This approach has a strong potential to apply to stochastic Volterra equations with fully dependent coefficients, as studied in \cite{alfonsi2024}. However, establishing the regularity of the path-dependent Feynman-Kac formula in this setting would be significantly more technical.
\medskip

We now prove Lemma \ref{lem:volterra_processes}. We introduce the flow processes  
\begin{equation*}
V_{T}^{t,x} = x + \int_t^T K(T-u) dW_u, \qquad  \bvol_{T}^{t,x} = x + \int_t^T \ak(T-u) dW_u
\end{equation*}
and the $(\mathcal{F}_t)_{t \in [0,T]}$-martingales
\begin{equation*}
\Theta^t_T = \int_0^t K(T-u) dW_u, \qquad \btheta^t_T = \int_0^t \ak(T-u) dW_u, 
\end{equation*}
so that 
\begin{align*}
V_T = \btheta^t_T + \int_t^T K(T-u) dW_u,\qquad
\bvol_T = \btheta^t_T + \int_t^T \ak(T-u) dW_u. 
\end{align*}
Considering the martingale $Y_t = \EE[\phi(V_T)/ \mathcal{F}_t]$ and using the above orthogonal decompositions, we observe that 
\[Y_t = \EE[ \phi( V_T^{t, \Theta^t_T} ) / \mathcal{F}_t] = u(t,\Theta^{t}_{T})\]
where the map $x\mapsto u(t,x) = \EE[ \phi( V_T^{t, x} )] = \EE[ \phi( x + \int_t^T K(T-u) dW_u)]$  inherits from $\phi$ the regularity in $x$, and is twice differentiable, with the representation
\begin{equation*}
\p_x u(t,x) = \EE[\phi'(V_T^{t,x})], \quad \pp_{xx} u(t,x) = \EE[\phi''(V_T^{t,x})].
\end{equation*} 
By formally applying the It\^{o} formula to the martingale  $Y_t = u(t,\Theta^{t}_{T})$, we immediately deduce that  $(t,x)\mapsto u(t,x)$ may satisfy the Kolmogorov backward equation
\begin{align}\label{eq:PDE_for_V}
\p_t u(t,x) + \frac{1}{2} K(T-t)^2 \pp_{xx} u(t,x) = 0,
\end{align}
with terminal condition $u(T,x) = \phi(x)$ for every $x \in \RR$. This Cauchy problem has a unique classical solution on rather standard assumptions, at least when $K$ is not singular on $[0,T]$ (see  e.g, \citet[Chap. 6]{friedman1975stochastic}).

The time in-homogeneous PDE \eqref{eq:PDE_for_V}, already introduced in \citet[Section 2.2]{viens2019}, allows a description of the marginals of the non-Markov process $\vol$ at a given time, but is not enough to track the law of the whole process $t \mapsto \vol_t$. 
But as in the Markov case, it is nevertheless enough to analyse the weak error between $\phi(\vol_T)$ and $\phi(\bvol_T)$, which depends only on the marginal of $\vol$ at final time $T$. 
Indeed, applying the It\^o formula to the martingale $(\btheta_T^t)_{t \in [0,T]}$ and the Kolmogorov PDE, we have
\begin{equation} \label{eq:ito_formula_btheta_finaltime}
\begin{aligned}
&  \EE[\phi(\bvol_T)] - \EE[\phi(\vol_T)] =   \EE[u(T,\btheta^T_T) - u(0,\btheta^0_T)] \\
&\qquad = \EE \left[ \int_0^T \p_t u(t,\btheta^t_T) dt + \int_0^T \p_x u(t,\btheta^t_T) d\btheta^t_T +  \frac{1}{2} \int_0^T \pp_{xx} u(\btheta^t_T) d\langle \btheta^\cdot_T\rangle_t \ \right] \\
& \qquad =  \EE \Bigg[ - \frac{1}{2} \int_0^T K^2(T-t) \pp_{xx} u(\btheta^t_T) dt + \int_0^T \p_x u(t,\btheta^t_T) d\btheta^t_T  + \frac{1}{2} \int_0^T \pp_{xx} u(\btheta^t_T) \ak^2(T-t) dt \ \Bigg].
\end{aligned}
\end{equation}
The process $\left( \int_0^T \p_x u(t,\btheta^t_T) d\btheta^t_T \right)_{t \in [0,T]}$ is a local martingale and the expectation of its quadratic variation at final time is given by
\begin{align*}
\EE \left[ \left\langle \int_0^{\cdot} \p_x u(t,\btheta^t_T) d\btheta^t_T \right\rangle_T  \right] 
= & \EE \left[ \int_0^T \ak^2(T-t) \  \p_x u(t,\btheta^t_T)^2 dt \right] 
= \int_0^T \ak^2(T-t) \ \big(\EE[\phi'(V_T^{t,\btheta^t_T})]\big)^2 dt .
\end{align*}
Using Jensen inequality and the growth control hypothesis $|\phi'(x)| \leqs 1 + |x|^{\kappa_{\phi}}$ and the inequality $(x+y)^2 \leq 2(x^2+y^2)$ we have that 
\begin{equation*}
  \big(\EE[ \phi'(V_T^{t,\btheta^t_T})]\big)^2 \leqs 2 \bigg(1+ \EE[(V_T^{t,\btheta^t_T})^{2\kappa_{\phi}}]\bigg).
\end{equation*}
Since the random variable $V_T^{t,\btheta^t_T} = \int_0^t \ak(T-u) dW_u + \int_t^T K(T-u) dW_u$ has a centred Gaussian distribution with variance $\sigma^2 = \int_0^t \ak(T-u)^2 du + \int_t^T K(T-u)^2 du$, we have (considering without loss of generality $\kappa_\phi$ as a positive integer), 
\begin{equation*}
  \EE[(V_T^{t,\btheta^t_T})^{2 \kappa_{\phi}}] = \left( \int_0^t \ak(T-u)^2 du + \int_t^T K(T-u)^2 du \right)^{2 \kappa_{\phi}} \frac{2^{\kappa_{\phi}} }{\sqrt{\pi}} \ \Gamma(\kappa_{\phi} + \frac{1}{2}).
\end{equation*}
Then assuming $\ak$ can be bounded by $K$ in the sense that there exists $c_{K,\ak} >0$ such that $\ak(r) \leq c_{K,\ak} K(r)$ for every $r \in [0,T]$, we obtain the estimate
\begin{equation*}
  \big(\EE[ \phi'(V_T^{t,\btheta^t_T})]\big)^2 \leqs 2 (1+c_{K,\ak})^{2\kappa_\phi} \left(1+ \frac{2^{\kappa_{\phi}}}{\sqrt{\pi}} \Gamma(\kappa_{\phi} + \frac{1}{2}) \left( \int_0^T K(T-u)^2 du \right)^{2 \kappa_{\phi}}\right).
\end{equation*}
Since $K \in L^2([0,T])$ it follows that $\sup_{t \in [0,T]} \EE[\phi'(V_T^{t,\btheta^t_T})]^2 < + \infty$ and the process $\left( \int_0^T \p_x u(t,\btheta^t_T) d\btheta^t_T \right)_{t \in [0,T]}$ is a square integrable martingale. We can hence rewrite the identity \eqref{eq:ito_formula_btheta_finaltime} as
\begin{equation} \label{eq:error_expansion_volterra_process}
  \EE[\phi(\bvol_T)] - \EE[\phi(\vol_T)] = \frac{1}{2} \int_0^T (\ak(T-t)^2 - K(T-t)^2) \ \EE[\phi''(V_T^{t,\btheta^t_T})] dt.
\end{equation}
Following the same arguments as above, it is easy to show that $\sup_{t \in [0,T]} \EE[\phi''(V_T^{t,\btheta^t_T})] < + \infty$. Then, a simple triangle inequality ensures the existence of a constant $C_{T,K,\ak,\phi} >0$ leading to \eqref{eq:weak_error_volterra_process}. 

This proof is not entirely rigorous in the case where the kernels $K$, $\ak$ are only in $L^2([0,T])$ and are not necessarily uniformly bounded.  In general, the $\ak$ approximation eliminates the singularity by truncation, as in the examples presented in the introduction.    It is then sufficient to use the well-posed heat equation in $\ak$, having the solution $x\mapsto \overline{u}(t,x) = \EE[ \phi( \bvol_T^{t, x} )] = \EE[ \phi( x + \int_t^T \ak(T-u) dW_u)]$ and apply the Itô  formula to $ \EE[\overline{u}(T,\Theta^T_T) - \overline{u}(0,\Theta^0_T)] $ that leads to the same conclusion.

\section{Main results}\label{sec:main}

We consider the integrated Volterra process $(\bitgr, \bvol)$, associated to a modified kernel $\ak$ and given by
\begin{equation}\label{eq:modified_integrated_volterra}
\begin{aligned}
\bitgr_t = & X_0 + \int_0^t b(s,\bvol_s) ds + \int_0^t  \sigma(s,\bvol_s) d B_s \\
\bvol_t = & \int_0^t \ak(t-u) dW_u,
\end{aligned}
\end{equation}
approximating the integrated Volterra process \eqref{eq:integrated_model}. 

Let $\func{\phi}{\RR}{\RR}$ be a test function. We are now interested in the weak error between $\phi(\itgr_T)$ and $\phi(\bitgr_T)$. We will make use of the following assumptions.

\begin{hyp}{\hspace{-0.1cm}\Blue{\bf  \textrm Regularity of $\phi$, $b$ and $\sigma$. (\ref{hyp:regularity}).}}  \makeatletter\def\@currentlabel{{\bf\textrm{H}}$_{\mbox{\scriptsize\bf\textrm{Reg}}}$}\makeatother   \label{hyp:regularity}
$\phi \in C^4(\RR,\RR)$ and there exists $\kappa_{\phi} > 0$ such that $|f(x)| \leqs (1+|x|^{\kappa_{\phi}})$ for $f \in \{ \phi, \phi', \phi'', \phi''', \phi''''\}$ and $x \in \RR$.
  One has $b,\sigma  (t,\cdot) \in C^3(\RR,\RR)$ for almost every $t \in [0,T]$ and there exists $\nu_{b}, \nu_{\sigma} > 0$ such that $|f(s,x)| \leqs 1+e^{\nu_{b}(x+s)}$ for $f \in \{ b, \p_{x} b , \p_{xx}^2 b  \}$ and $(s,x) \in [0,T] \times \RR$, and such that $|g(s,x)| \leqs 1+e^{\nu_{\sigma}(x+s)}$ for $g \in \{ \sigma, \p_{x} \sigma , \p_{xx}^2  \sigma  \}$ and $(s,x) \in [0,T] \times \RR$. In addition, one has $X_0 \in L^p(\Omega)$ for every $p \geq 1$. Regarding the time variable, for almost every $x \in \RR$, assume that $h(\cdot,x) \in C^0([0,T],\RR)$ for $h \in \{b,\p_{x} b ,\p_{xx}^2 b ,\sigma,\p_{x} \sigma ,\p_{xx}^2  \sigma \}$.
\end{hyp}

\begin{hyp}{\hspace{-0.1cm}\Blue{\bf  \textrm Condition on $K$ and $\ak$. (\ref{hyp:kernel}).}}
\makeatletter\def\@currentlabel{{\bf\textrm{H}}$_{\mbox{\scriptsize\bf\textrm{Ker}}}$}\makeatother
\label{hyp:kernel}
$K \in L^2([0,T],\RR_+)$ and $\ak \in L^2([0,T],\RR_+)$. The function $t \mapsto K(t)$ is non-increasing, continuous on $(0,T]$ and differentiable almost everywhere on $(0,T]$, and there exists $H \in (0,\frac{1}{2})$ such that for any $t \in (0,T]$,
\begin{equation*}
K(t) \leqs t^{H-\frac{1}{2}}, \quad |K'(t)| \leqs t^{H-\frac{3}{2}}.
\end{equation*}
The function $\ak$ is continuous on $(0,T]$, and there exists $c_{K,\ak} > 0$ such that
\begin{align*}
\text{for all} ~ r \in [0,T],  \quad \ak(r) \leq c_{K,\ak} \ K(r).
\end{align*}
\end{hyp}

\begin{theorem}[Weak comparison] \label{thm:weakcvg_integrated}
Assume that \ref{hyp:regularity} and \ref{hyp:kernel} hold. Then, considering the intergated Volterra processes \eqref{eq:integrated_model} and \eqref{eq:modified_integrated_volterra}, $(X,V)$ and $(\bitgr,\bvol)$, 
\begin{equation*}
    | \EE[\phi(\itgr_T)] - \EE[\phi(\bitgr_T)] | \leq C_K \int_0^T \Big\{ \|{K - \ak}\|_{L^1([0,t])} + \|{K^2 - \ak^2}\|_{L^1([0,t])} \Big\} dt
\end{equation*}
with \begin{equation*}
C_K \coloneq  C (1 +  \indic{\rho > 0} (\|K\|_{L^2([0,T])}+ \|K\|^2_{L^2([0,T])}) ) \ (1+\exp(C \ \|{K} \|_{L^2([0,T])}^2 \vee \|{\ak} \|_{L^2([0,T])}^2)),
\end{equation*}
where $C > 0$ does not depend on $K$.
\end{theorem}
\begin{remark}~
\begin{itemize}
\item[(i)] {\bf On  $\phi$ just in  $C^3(\RR,\RR)$. }   If we only have $\phi \in C^3(\RR,\RR)$, the result still holds replacing the constant $C_K$ by
  \begin{equation*}
    C_K \coloneq C (1 + \|K\|_{L^2([0,T])} + \indic{\rho > 0} \|K\|^2_{L^2([0,T])}) \ (1+\exp(C \ \|{K} \|_{L^2([0,T])}^2 \vee \|{\ak} \|_{L^2([0,T])}^2)).
  \end{equation*}
\item[(ii)]  {\bf On the rate of convergence. } 
The above error estimate contains two terms. The first one $\int_0^T \|{K - \ak}\|_{L^1([0,t])}dt$ usually behaves better than the second $\int_0^T \|{K^2 - \ak^2}\|_{L^1([0,t])}dt$.  If we think, for instance, of the  truncation approximation example \eqref{eq:kernel_appro_truncated}, $\|{K - \ak}\|_{L^1([0,t])}\leq \frac{1}{H+\frac{1}{2}} \tau^{H+\frac{1}{2}}$, whereas    $\|{K^2 - \ak^2}\|_{L^1([0,t])}\leq  \frac{1}{{2H}} \tau^{2H}$, that makes a huge difference between the two contributions  when $H$ is chosen small. Note that the $L^2$ contribution is already present when comparing the Volterra processes in Lemma \ref{lem:volterra_processes}, where the Gaussian case applies, simplifying the analysis.
The error decomposition \eqref{eq:error_decompos} in Section \ref{sec:proof_thm_1} shows that the contribution $\| (K^2) - (\ak^2) \|_{L^1([0,T])}$ may vanish  under the very  restrictive conditions that $b$ is a linear drift and $\sigma$ a constant diffusion coefficient in the definition of the integrated Volterra process \eqref{eq:integrated_model}.

\end{itemize}
\end{remark} 

\section{Preliminary tools, functional Itô formula and  Path-Dependent PDES}\label{sec:PPDES}

We introduce some preliminary notations in order to define path derivatives and the functional Itô formula, as well as to introduce the path-dependent PDEs. The content of this section is mostly taken from \cite{viens2019} and \cite{bonesini2023rough}, with some slight adaptations and an enhancement of the computation of the path derivatives.

We start with the introduction of the notion of path derivatives. We fix a real $t \in [0,T]$, and define
\begin{align*}
  & \Xx \coloneq C^0([0,T],\RR),\qquad \Xx_t \coloneq C^0([t,T],\RR), \qquad \overline{\Xx} \coloneq D^0([0,T],\RR); \\
  & \Lambda \coloneq [0,T] \times \RR \times C^0([0,T],\RR); \\
  & \bar{\Lambda} \coloneq \left\{(t,x,\omega) \in [0,T] \times \RR \times D^0([0,T],\RR) \ ; \ \omega_{|_{[t,T]}} \in \Xx_t \right\}; \\
  & \|\omega\|_{[0,T]} \coloneq \sup_{t \in [0,T]} |\omega_t|, \qquad d_{\bar{\Lambda}}( \ (t,x,\omega),(t',x',\omega') \ ) \coloneq |t-t'| + |x-x'| + \| \omega - \omega' \|_{[0,T]}.
\end{align*}
Let $C^0(\bar{\Lambda}) \coloneq C^0(\bar{\Lambda},\RR)$ be the space of functions $\func{u}{\bar{\Lambda}}{\RR}$ such that $u$ is continuous with respect to the topology induced by the distance $d_{\bar{\Lambda}}$. We define the time derivative of a function $u \in C^0(\bar{\Lambda})$ at $(t,x,\omega) \in \bar{\Lambda}$ as the following limit (if it exists):
\begin{equation*}
\p_t u(t,\omega) = \lim_{\delta \to 0} \frac{u(t+\delta,\omega) - u(t,\omega)}{\delta}. 
\end{equation*}
Let $(t,x) \in [0,T] \times \RR$ and $\eta \in \Xx_t$.
We define the path derivative of $u$ at $\omega \in \Xx$ in the direction $\eta$ by
\begin{equation*}
\av{\p_{\omega} u(t,x,\omega),\eta} = \lim_{\varepsilon \to 0} \frac{u(t,x,\omega \ind_{[t,T]} + \varepsilon \eta) - u(t,x,\omega \ind_{[t,T]})}{\varepsilon},
\end{equation*}
if the limit exists. For any $r < t$ and $\eta \in \Xx_r$,  we fix the convention that
\begin{equation*}
\av{\p_{\omega} u(t,x,\omega),\eta} = \av{\p_{\omega} u(t,x,\omega),\eta \ind_{[t,T]}}.
\end{equation*}
Similarly, we define the second path derivative of $u$ at point $(t,x,\omega) \in \bar{\Lambda}$, in the directions $\eta, \zeta \in \Xx_t$ by
\begin{equation*}
    \av{\p^2_{\omega} u(t,x,\omega ),(\eta,\zeta)} = \lim_{\varepsilon \to 0} \frac{ \f{u(t,x,\omega \ind_{[t,T]} + \varepsilon \zeta)}{\eta} - \f{u(t,x,\omega \ind_{[t,T]})}{\eta} }{\varepsilon},
\end{equation*}
if the limit exists. Note that the path derivatives are actually defined as G\^{a}teaux-derivatives with respect to $\omega \ind_{[t,T]}$.

\begin{definition} \label{def:regularity}
    We say that $u \in C^{2,2}(\bar{\Lambda}) \subset C^0(\bar{\Lambda})$ if the derivatives $\p_x u$, $\p_{xx} u$, $\p_{\omega}$, $\p_{\omega \omega} u$ and $\p_{\omega}( \p_x u )$ exist and are continuous functions (at fixed direction) on $\bar{\Lambda}$ with respect to the distance $d_{\bar{\Lambda}}$, and if the maps $\eta \mapsto \av{\p_{\omega} u(t,x,\omega),\eta}$, $\zeta \mapsto \av{\p^2_{\omega} u(t,x,\omega),(\eta,\zeta)}$ and $\eta \mapsto \av{\p_{\omega}( \p_x u)(t,x,\omega),\eta}$ are linear and continuous on $\Xx$.
\end{definition}

One may observe that due to the definition of $d_{\bar{\Lambda}}$, it is enough to show that the derivatives of $u$ are continuous respectively with respect to $t$, $x$ and $\omega$ separately to ensure that they are continuous under $d_{\bar{\Lambda}}$ on $\bar{\Lambda}$. A sufficient condition for the linearity and the continuity of the path derivatives with respect to the direction is that $u$, $\p_{\omega} u$ and $\p_x u$ are Fr\'{e}chet-differentiable. Note that Fr\'{e}chet and G\^{a}teaux-differentiability are not equivalent in this context since the space $\Xx$ has infinite dimension. However, after a careful check of the proofs of \cite{viens2019} and \cite{bonesini2023rough}, the Fr\'{e}chet-differentiability is not necessary as long as the conditions of Definition~\ref{def:regularity} are satisfied. For this reason, we have chosen to use the notion of Hadamard derivative in the proofs (see the precise definition in Appendix \ref{sec:appendix_hadam}), which allows the chain rule to be retained in calculations while being better suited to the framework of our study (see, for instance, the proof of Proposition \ref{prop:representation_first_order} below).
We will use the following growth conditions on the derivatives of $u$.

\begin{definition}
  Let $u\in C^{2,2}(\bar{\Lambda})$.
  \begin{enumerate}[label=(\roman*)]
    \item We say that $u$, $\p_x u$ or $\p^2_{xx} u$ has controlled growth if there exists a constant $q > 0$ such that
    \begin{equation*}
    \abs{\p^i_{x^i} u(t,x,\omega)}\lesssim (1 + |x|^q + e^{q \| \omega \|_{[0,T]}}), \ \text{ for all  } (t,x,{\omega})\in\bar{\Lambda}, \quad \text{for }i=0, \text{ resp.  } 1 \text{ or } 2. 
    \end{equation*}
    \item We say that $\partial_{{\omega}} u$ or $\partial_{{\omega}}( \p_x u )$  has controlled growth if there exists a constant $q > 0$ such that
    \begin{equation*}
    \abs{\langle \partial_{{\omega}} \p^{j}_{x^j} u(t,x,{\omega}),\eta\rangle}\lesssim (1 + |x|^q + e^{q \| \omega \|_{[0,T]}}) \norm{\eta \one{{[t,T]}}}_{[0,T]}, \ \text{ for all  } (t,x,{\omega})\in\bar{\Lambda}, \ \eta \in \Xx, \quad \text{for }j=0, \text{ resp. } 1. 
    \end{equation*}
    \item We say that $\partial_{{\omega}{\omega}} u$ has controlled growth if there exists a constant $q > 0$ such that
    \begin{equation*}
    \abs{\langle \partial^2_{{\omega}{\omega}} u(t,x,{\omega}),(\eta,\zeta)\rangle}\lesssim (1 + |x|^q + e^{q \| \omega \|_{[0,T]}}) \norm{\eta \one{{[t,T]}}}_{[0,T]}\norm{\zeta \one{{[t,T]}}}_{[0,T]}, \ \text{ for all  } (t,x,{\omega})\in\bar{\Lambda}, \ \eta, \zeta \in \Xx.
    \end{equation*}
  \end{enumerate}
  When $u$ and all the derivatives above have controlled growth, we write $u \in C^{2,2}_+(\bar{\Lambda})$.
\end{definition}

\subsection{Orthogonal decomposition and flow processes}
For $(t,s) \in [0,T]^2$ with $t \leq s$, we consider the dual-time martingales
\begin{equation} \label{eq:martingales_a_deu\itgr_Temps}
\begin{aligned}
\Theta^t_s = & \int_0^t K(s-r) dW_r \ \text{ and } \ \btheta^t_s = \int_0^t \ak(s-r) dW_r, 
\end{aligned}
\end{equation}
so that we have the orthogonal decompositions
\begin{align*}
\vol_s = \Theta^t_s + \int_t^s K(s-r) dW_r \ \text{ and } \ \bvol_s = \btheta^t_s + \int_t^s \ak(s-r) dW_r.
\end{align*}
For $(t,x,\omega) \in [0,T] \times \RR \times C^0([0,T],\RR)$, we define the stochastic flow
\begin{equation} \label{eq:flow}
\begin{aligned}
    \itgr_T^{t,x,\omega} = & x + \int_t^T b(s,\vol_s^{t,\omega}) ds + \int_t^T \sigma(s,\vol_s^{t,\omega}) dB_s \\
    \vol_s^{t,\omega} = & \omega_s + \int_t^s K(s-r) dW_r.
\end{aligned}
\end{equation}
In the following lemma, we obtain control, uniformly in time, of the moments of the stochastic flow \eqref{eq:flow} started at fixed $(x,\omega)$, which will allow us to prove the regularity of $(x,\omega) \mapsto \EE[\phi(X_T^{t,x,\omega})]$. We also obtain a uniform in time control of the moments of the stochastic flow started at the approximated process $(\bitgr_t,\btheta^t)$, which is needed to get upper bounds for the stochastic terms appearing after the use of the functional Itô formula in our weak error analysis. The proof utilises standard arguments, which are detailed in Appendix \ref{sec:proof_controle_des_moments} for clarity and completeness. 
\begin{lemma} \label{lm:controle_des_moments} 
Assume \ref{hyp:regularity} and \ref{hyp:kernel}. Let $p \geq 1$ and $(x,\omega) \in \RR \times C^0([0,T],\RR)$. There exists positive constants $m^{(1)}_{p}, m^{(2)}_{p}, \overline{m}^{(1)}_{p}$ and $\overline{m}^{(2)}_{p}$ that does not depend on $K$ and $T$, such that
\begin{equation*}
\begin{aligned}
 \sup_{t \in [0,T]} \EE[ | \itgr_T^{t,x,\omega} |^p] \leqs \ |x|^p + e^{m^{(1)}_{p} (\|K\|_{L^2([0,T])}^2 + \| \omega \|_{[0,T]} ) },
&  \qquad \sup_{t \in [0,T]} \sup_{s \in [t,T]} \EE[ \exp(p \vol_s^{t,\omega} )] \leqs e^{m^{(2)}_{p} (\|K\|_{L^2([0,T])}^2 + \| \omega \|_{[0,T]})}, \\
  \sup_{t \in [0,T]} \EE[ | \itgr_T^{t,\bitgr_t,\btheta^t} |^p] \leqs \ \EE[|X_0|^p] + e^{\overline{m}^{(1)}_{p} \|K\|_{L^2([0,T])}^2 },
& \qquad \sup_{t \in [0,T]} \sup_{s \in [t,T]} \EE[ \exp(p \vol_s^{t,\btheta^t} )] \leqs e^{\overline{m}^{(2)}_{p} \|K\|_{L^2([0,T])}^2 },
\end{aligned}
\end{equation*}
where the multiplicative constant hidden in the use of the symbol $\leqs$ may depend on $T$, $p$, $\nu_b$, $\nu_{\sigma}$, $\kappa_{b}$ and $\kappa_{\sigma}$ but not on $K$.
\end{lemma}
In what follows, we are interested in the regularity of the function $u(t,x,\omega) = \EE[\phi(\itgr_T^{t,x,\omega})]$. In particular, we would like to show that the derivatives $\p_t u$ exists and that $u$ belongs to $C^{2,2}_+(\bar{\Lambda})$, to have access to the two main tools required for our error analysis, namely:
\begin{enumerate}
\item \textbf{The functional Itô formula}:
\begin{align*}
u(t,\bitgr_t,\btheta^t) = 
& \ u(0,X_0,0) + \int_0^t \p_t u(s,\bitgr_s,\btheta^s) ds + \int_0^t \p_x u(s,\bitgr_s,\btheta^s) b(s,\bvol_s) ds \\ 
& + \frac{1}{2} \int_0^t \p_{xx} u(s,\bitgr_s,\btheta^s) \sigma^2(s,\bvol_s) ds + \int_0^t {\rho \sigma(s,\bvol_s)} \ \av{ \p_{\omega}( \p_x u )(s,\bitgr_s,\btheta^s), \ak(\cdot - s)} ds \\ 
& + \frac{1}{2} \int_0^t \av{\p_{\omega \omega} u(s,\bitgr_s,\btheta^s) , (\ak(\cdot - s),\ak(\cdot - s))} ds \\
& + \int_0^t \p_x u(s,\bitgr_s,\btheta^s) \sigma(s,\bvol_s) dB_s + \int_0^t \av{\p_{\omega} u(s,\bitgr_s,\btheta^s),\ak(\cdot - s)} dW_s.
\end{align*}
  
\item \textbf{The path-dependent Feynman-Kac formula}: $u$ satisfies 
\begin{align*}
\p_t u(t,x,\omega) & + \p_x u(t,x,\omega) b(t,\omega_t) + \frac{1}{2} \p_{xx} u(t,x,\omega) \sigma^2(t,\omega_t) \\
& + \frac{1}{2} \ff{u(t,x,\omega)}{K(\cdot-t)}{K(\cdot - t)}  + \rho \ \sigma(t,\omega_t) \f{(\p_x u)(t,x,\omega)}{K(\cdot - t)} = 0,
\end{align*}
with terminal condition $u(T,x,\omega) = \phi(x)$, for any $(t,x,\omega) \in \bar{\Lambda}$.
\end{enumerate}
Note that in both of the identities above, the path derivatives of $u$ are taken in the directions $\eta, \zeta = \ak(\cdot - s)$ and $\eta, \zeta = K(\cdot - t)$, respectively. These directions do not belong to $\Xx_t$ since they are not continuous on $[t,T]$, due to the singularity of $K$ (and possibly $\ak$) at zero. 
For that reason, we start  proving that  $u \in C^{2,2}_+(\bar{\Lambda})$ in Proposition \ref{prop:representation_first_order}, stating the  regularity of $u$ in the case of continuous directions in Section \ref{sec:regularity_continuous_directions}.
Then we will extend the results to the case of singular directions in Section \ref{sec:regularity_singular_directions}, proving that $u \in C^{0,2,2}_{+, \alpha} (\Lambda)$ for some $\alpha \in (0,1)$ (see Definition~\ref{def:Cplusalpha} below), which induces that any extension $\tilde{u}$ of $u$ on $\bar{\Lambda}$ satisfies $\tilde{u} \in C^{2,2}_{+}(\bar{\Lambda})$ (see \citet[Proposition 3.7]{viens2019}).

\subsection{Regularity of $\omega \mapsto \EE[\phi(\itgr_T^{t,x,\omega})]$ -- the case of continuous directions for $\p_\omega$} \label{sec:regularity_continuous_directions}

When computing path derivatives in practice, it is often helpful to have access to a notion of chain rule. In the context of $L^p(\Omega)$ spaces, the chain rule is not straightforward since it does not hold with G\^{a}teaux derivatives, and the functions involved are not necessarily Fr\'{e}chet-differentiable at each stage. For this reason, we will rely on the concept of Hadamard derivatives, which is the minimum requirement for the chain rule to apply. We recall the definitions of Hadamard differentiability and the chain rule property in Appendix \ref{sec:appendix_hadam}.

The following lemma gives the first and second Hadamard derivatives in the space $L^p(\Omega)$ of the flow process $\itgr_T^{t,x,\omega}$.
\begin{lemma} \label{lm:gateaux_derivative_X}
Assume  \ref{hyp:regularity} and \ref{hyp:kernel}.
Fix $(t,x) \in [0,T] \times \Xx_t$  and $p\geq 1$. 
\begin{enumerate}[label=(\roman*)]
\item The function $\omega \mapsto X_T^{t,x,\omega}$ is Hadamard-differentiable in $L^p(\Omega)$, and the Hadamard derivative at $\omega \in \Xx_t$ in the direction $\eta \in \Xx_t$ is given by
\begin{equation*}
\av{\partial_{\omega} X_T^{t,x,\omega},\eta} = \int_t^T \p_{x} b (s,\vol_s^{t,\omega}) \eta_s ds + \int_t^T  \p_{x} \sigma (s,\vol_s^{t,\omega}) \eta_s dB_s.
\end{equation*}
        
\item The function $\omega \mapsto \f{X_T^{t,x,\omega}}{\eta}$ is Hadamard-differentiable in $L^p(\Omega)$, and the Hadamard derivative at $\omega \in \Xx_t$ in the direction $\zeta \in \Xx_t$ is given by
\begin{equation*}
\av{\partial^2_{\omega\omega} X_T^{t,x,\omega},(\eta,\zeta)} = \int_t^T \p_{xx}^2 b (s,\vol_s^{t,\omega}) \eta_s \zeta_s ds + \int_t^T  \p_{xx}^2  \sigma (s,\vol_s^{t,\omega}) \eta_s \zeta_s dB_s.
\end{equation*}
\end{enumerate}
\end{lemma}

\begin{proof}
We only prove the first point, since the second one uses the same arguments (along with the fact that $b$ and $\sigma$ are in $C^3(\RR,\RR)$ with polynomial growth derivatives). 
Let $\omega \in \Xx_t$ and $\eta \in \Xx_t$. 
We consider a sequence $(\eta^{(n)})_{n \in \NN}$ of elements of $\Xx_t$ such that $\| \eta^{(n)} - \eta \|_{[0,T]} \to 0$ as $n \to \infty$, and a sequence of real numbers $(\varepsilon_n)_{n \in \NN}$ such that $\varepsilon_n \to 0$ as $n \to \infty$. 
Denoting formally
\begin{equation*}
\f{X_T^{t,x,\omega}}{\tau} \coloneq \int_t^T \p_x b(s,\vol_s^{t,\omega}) \tau_s ds + \int_t^T \p_x \sigma(s,\vol_s^{t,\omega}) \tau_s dB_s,
\end{equation*}
for any $\tau \in \Xx_t$, we aim to show that
\begin{equation*}
\left| \frac{1}{\varepsilon_n} \left( X_T^{t,x,\omega + \varepsilon_n \eta^{(n)}} - X_T^{t,x,\omega} \right) - \f{X_T^{t,x,\omega}}{\eta} \right| \ \underset{n \to +\infty}{\xrightarrow{L^p(\Omega)}} \ 0.
\end{equation*}
We make use of a pivot term:
\begin{equation} \label{eq:decomposition_hadamard_derivative_X}
\begin{aligned}
\frac{1}{\varepsilon_n} \left( X_T^{t,x,\omega + \varepsilon_n \eta^{(n)}} - X_T^{t,x,\omega} \right) 
- \f{X_T^{t,x,\omega}}{\eta} 
= &\frac{1}{\varepsilon_n} \left( X_T^{t,x,\omega + \varepsilon_n \eta^{(n)}} - X_T^{t,x,\omega} \right) - \f{X_T^{t,x,\omega}}{\eta^{(n)}} \\ 
& + \f{X_T^{t,x,\omega}}{\eta^{(n)}} - \f{X_T^{t,x,\omega}}{\eta}.
\end{aligned}
\end{equation}
By definition, 
\begin{align*}
\f{X_T^{t,x,\omega}}{\eta^{(n)}} - \f{X_T^{t,x,\omega}}{\eta} = & \int_t^T \p_x b(s,\vol_s^{t,\omega}) (\eta_s^{(n)} - \eta_s) ds + \int_t^T \p_x \sigma(s,\vol_s^{t,\omega}) (\eta_s^{(n)} - \eta_s) dB_s. 
\end{align*}
Applying  Minkoswki's integral inequality and Burkhölder-Davis-Gundy (BDG) inequality with constant $c_p$, we get
\begin{equation}\label{eq:continuite_hadamard_derivative_X}
\begin{aligned}
\| \f{X_T^{t,x,\omega}}{\eta^{(n)}} - \f{X_T^{t,x,\omega}}{\eta} 
\|_{L^p(\Omega)} \leq & \int_t^T \| \p_x b(s,V_s^{t,\omega}) \|_{L^p(\Omega)} \ |\eta_s^{(n)} - \eta_s| ds \\ 
& + c_p^\frac{1}{p} \EE \left[ \left( \int_t^T  (\p_x\sigma)^2(s,V_s^{t,\omega}) 
\ (\eta_s^{(n)} - \eta_s )^2 ds \right)^{\frac{p}{2}} \right]^{\frac{1}{p}}.
\end{aligned}
\end{equation}
Using the regularity assumptions \ref{hyp:regularity} and the moment bound from Lemma \ref{lm:controle_des_moments}, we get that the first term in right-hand side of the inequality \eqref{eq:continuite_hadamard_derivative_X} is bounded by  $C\| \eta^{(n)} - \eta \|_{[0,T]}$, which goes to $0$ as $n \to +\infty$. For the second term, if $p \geq 2$ we may use Minkoswki's integral inequality again to get
\begin{align*}
\EE \left[ \left( \int_t^T  (\p_x\sigma)^2(s,V_s^{t,\omega}) \  (\eta_s^{(n)} - \eta_s )^2 ds \right)^{\frac{p}{2}} \right]^{\frac{1}{p}} 
\leq  \left( \int_t^T \| (\p_x\sigma)^2(s,V_s^{t,\omega})\|_{L^{p/2}(\Omega)}( \ \eta_s^{(n)} - \eta_s )^2 ds \right)^\frac{1}{2},
\end{align*}
while if $p \in [1,2)$, Jensen inequality applied to the concave function $x \mapsto x^{\frac{p}{2}}$ gives
\begin{align*}
\EE \left[ \left( \int_t^T  (\p_x\sigma)^2(s,V_s^{t,\omega}) \ (\eta_s^{(n)} - \eta_s )^2 ds \right)^{\frac{p}{2}} \right]^{\frac{1}{p}} 
\leq \left( \int_t^T \EE[(\p_x\sigma)^2(s,V_s^{t,\omega})] \ (\eta_s^{(n)} - \eta_s )^2 ds \right)^\frac{1}{2}.
\end{align*}
In both cases, we may use \ref{hyp:regularity} and Lemma \ref{lm:controle_des_moments} to get that the second term in right-hand side of \eqref{eq:continuite_hadamard_derivative_X} is bounded by $C \| \eta^{(n)} - \eta \|_{[0,T]}$, which goes to $0$ as $n \to +\infty$. We have thus shown that $\f{X_T^{t,x,\omega}}{\eta^{(n)}} \to \f{X_T^{t,x,\omega}}{\eta}$ in $L^p(\Omega)$ as $n \to +\infty$.

We now turn to the first term in the right-hand side of \eqref{eq:decomposition_hadamard_derivative_X}. We have
\begin{align*}
& \frac{1}{\varepsilon_n} \left( X_T^{t,x,\omega + \varepsilon_n \eta^{(n)}} - X_T^{t,x,\omega} \right) - \f{X_T^{t,x,\omega}}{\eta^{(n)}} \\
& = \frac{1}{\varepsilon_n}\int_t^T  \left( b(s,\vol_s^{t,\omega+\varepsilon_n \eta^{(n)}}) - b(s,\vol_s^{t,\omega}) - \p_{x} b (s,\vol_s^{t,\omega}) \varepsilon_n \eta^{(n)}_s \right) ds \\ 
& \quad + \frac{1}{\varepsilon_n}\int_t^T  \left( \sigma(s,\vol_s^{t,\omega+\varepsilon_n \eta^{(n)}}) - \sigma(s,\vol_s^{t,\omega}) - \p_{x} \sigma (s,\vol_s^{t,\omega}) \varepsilon_n \eta^{(n)}_s \right) dB_s.
\end{align*}
Using a Taylor expansion for $b$, we can write
\begin{align*}
b(s,\vol_s^{t,\omega+\varepsilon_n \eta^{(n)}}) - b(s,\vol_s^{t,\omega}) 
& =  (\vol_s^{t,\omega+\varepsilon_n \eta^{(n)}} - \vol_s^{t,\omega}) \int_0^1 \p_{x} b (s,\vol_s^{t,\omega} + \lambda (\vol_s^{t,\omega+\varepsilon_n \eta^{(n)}} - \vol_s^{t,\omega})) \ d\lambda, \\ 
& = \varepsilon_n \eta^{(n)}_s \int_0^1 \p_{x} b (s,\vol_s^{t,\omega} + \lambda \varepsilon_n \eta^{(n)}_s ) \ d\lambda
\end{align*}
and similarly for  $\p_{x} b $,
\begin{equation*}
\p_{x} b (s,\vol_s^{t,\omega} + \lambda \varepsilon_n \eta^{(n)}_s ) - \p_{x} b (s,\vol_s^{t,\omega}) 
=  \lambda \varepsilon_n \eta^{(n)}_s \int_0^1 \p_{xx}^2 b (s,\vol_s^{t,\omega} + \mu \lambda \varepsilon_n \eta^{(n)}_s ) \ d\mu.
\end{equation*}
The same argument holds for $\sigma$ and we deduce that
\begin{equation} \label{eq:gateaux_derivative_X}
\begin{aligned}
\frac{1}{\varepsilon_n} \left( X_T^{t,x,\omega + \varepsilon_n \eta^{(n)}} - X_T^{t,x,\omega} \right)
 - \f{X_T^{t,x,\omega}}{\eta^{(n)}}
& = \varepsilon_n \int_t^T (\eta^{(n)}_s)^2 \int_0^1 \int_0^1 \lambda \p_{xx}^2 b (s,\vol_s^{t,\omega} + \mu \lambda \varepsilon_n \eta^{(n)}_s ) d\mu \ d\lambda \ ds \\ 
& \quad + \varepsilon_n  \int_t^T (\eta^{(n)}_s)^2 \int_0^1 \int_0^1 \lambda \p_{xx}^2  \sigma (s,\vol_s^{t,\omega} + \mu \lambda \varepsilon_n \eta^{(n)}_s ) d\mu \ d\lambda \ dB_s.
\end{aligned}
\end{equation}

Let $p\geq 1$ and $\psi \in \{b,\sigma\}$. One has $(\eta^{(n)}_s)^2 \leq \sup_{n \in \NN} \| (\eta^{(n)})^2 \|_{[0,T]} < + \infty$, since $(\eta^{(n)})^2 \to (\eta)^2$, and using \ref{hyp:regularity}, we have
\begin{equation} \label{eq:control_gateaux_derivative}
\begin{aligned}
\EE[|\p^2_{xx}\psi (s,\vol_s^{t,\omega} 
+ \mu \lambda \varepsilon_n \eta_s^{(n)} )|^p] 
\leq & \EE[(1+\exp(\kappa_{\psi} (\vol_s^{t,\omega} 
+ \mu \lambda \varepsilon \eta_s^{(n)} + s) ) )^p] \\ 
& \leqs  1 + \EE[\exp(\kappa_{\psi} p (\vol_s^{t,\omega} 
+ \mu \lambda \varepsilon \eta_s^{(n)} + s))] \\ 
& \leqs   1 + \exp(\kappa_{\psi} p ( \sup_{n \in \NN} \{ \varepsilon_n \normt{\eta^{(n)}} \} + T)  ) \sup_{s\in[0,T]} \EE[\exp(\kappa_{\psi} p \vol_s^{t,\omega} )].
\end{aligned}
\end{equation}
The latter supremum is bounded thanks to Lemma \ref{lm:controle_des_moments} and since $\varepsilon_n \to 0$ and $\eta^{(n)} \to \eta$. We conclude using the $L^p(\Omega)$-Minkowski integral inequality on the time integral in \eqref{eq:gateaux_derivative_X} combined with \eqref{eq:control_gateaux_derivative}, as well as the Burkholder-Davis-Gundy inequality on the Brownian integral in \eqref{eq:gateaux_derivative_X} combined with \eqref{eq:control_gateaux_derivative}. This allows to obtain a uniform bound in $n$ for the $L^p(\Omega)$-norm of the two integrals in \eqref{eq:gateaux_derivative_X}, and the result follows by letting $n$ go to $+ \infty$.
\end{proof}
We may now check the regularity of the function $u(t,x,\omega) = \EE[\phi(X_T^{t,x,\omega})]$ and obtain a representation for the derivatives of $u$. This is realized in the following proposition.
\begin{proposition} \label{prop:representation_first_order}
  Let $t \in [0,T]$ and $\eta, \zeta \in \Xx_t$, and assume \ref{hyp:regularity}. The function $u:\bar{\Lambda} \to \RR$ defined by $u(t,x,\omega) = \EE[\phi(\flowx_T)]$ belongs to the space $C^{2,2}_+(\bar{\Lambda})$. Moreover, the following representations hold:
\begin{equation} \label{eq:probabilistic_representation}
  \begin{aligned}
    \p_x u(t,x,\omega) = & \EE[\phi'(\flowx_T)] \\
    \pp_x u(t,x,\omega) = & \EE[\phi''(\flowx_T)] \\
    \f{u(t,x,\omega)}{\eta} = & \EE[\phi'(\flowx_T) \f{X_T^{t,x,\omega}}{\eta} ] \\
    \f{ \p_x u(t,x,\omega)}{\eta} = & \EE[\phi''(\flowx_T) \f{X_T^{t,x,\omega}}{\eta} ] \\
    \ff{u(t,x,\omega)}{\eta}{\zeta} = & \EE[ \phi''(\flowx_T) \f{X_T^{t,x,\omega}}{\eta} \f{X_T^{t,x,\omega}}{\zeta}  ] + \EE[\phi'(\flowx_T) \ff{X_T^{t,x,\omega}}{\eta}{\zeta} ]. 
  \end{aligned}
\end{equation}
\end{proposition}
\begin{remark}
  In particular, the map $(\eta,\zeta) \mapsto \ff{u(t,x,\omega)}{\eta}{\zeta}$ is a bilinear symmetric form on $\Xx_t^2$. 
\end{remark}
\begin{proof}
The proof is organized into three parts. First, we show that $u$ is twice differentiable with respect to its variable $x$, twice Hadamard-differentiable with respect to its pathwise variable $\omega$, and that $\p_x u$ is itself Hadamard-differentiable with respect to $\omega$. Second, we show that the derivatives of $u$ are continuous with respect to the distance $d_{\bar{\Lambda}}$ defined at the beginning of Section \ref{sec:PPDES}. We will observe that these first two points ensure that $u \in C^{2,2}(\bar{\Lambda})$ and provide the expected probabilistic representations of the derivatives of $u$. Finally, in the third step, we verify that all derivatives of $u$ grow at most polynomially in $x$ and exponentially in $\omega$, thereby concluding that $u \in C^{2,2}_{+}(\bar{\Lambda})$.

\paragraph{1. Hadamard-differentiability and associated probabilistic representations. }
The twice differentiable nature of the function $x \mapsto u(t,x,\omega)$ for fixed $(t,\omega) \in [0,T] \times \Xx_t$ is established through standard arguments. Indeed, by applying Taylor's formula to the locally Lipschitz functions $\phi$ and $\phi'$ respectively, using the facts that $\phi'$ and $\phi''$ have at most polynomial growth and applying Lemma \ref{lm:controle_des_moments}, it is straightforward to deduce that $\p_x u(t,x,\omega) = \EE[\phi'(\flowx_T)]$ and $\pp_x u(t,x,\omega) = \EE[\phi''(\flowx_T)]$. Now, fix $t \in [0,T]$ and $x \in \RR$. 
We demonstrate that $\omega \mapsto u(t,x,\omega)$ is Hadamard-differentiable: in addition to establishing the existence of the path derivative $\p_{\omega} u$ in the sense of Definition \ref{def:regularity}, we show that the mapping $\eta \mapsto \f{u(t,x,\omega)}{\eta}$ is continuous (see, for example, \citet[Proposition 3.1]{shapiro1990}). The linearity of the latter  then follows immediately by using the obtained probabilistic representation of $\p_{\omega} u$. Applying Lemma \ref{lm:hadamard_differentiability} with $f = \phi$, $p = \kappa_{\phi}$, $r=1$, and $q = 2 \kappa_{\phi} \vee 4$, the extension of $\phi$ to the spaces $L^q(\Omega)$ and $L^1(\Omega)$ is Hadamard-differentiable. Moreover, due to Lemma \ref{lm:gateaux_derivative_X}, the function $\omega \mapsto X_T^{t,x,\omega}$ is also Hadamard-differentiable in $L^q(\Omega)$.
This allows us to apply the chain rule from Proposition \ref{prop:chain_rule}, which shows that $\omega \mapsto \phi(X_T^{t,x,\omega})$ is Hadamard-differentiable in $L^1(\Omega)$, and that the Hadamard derivative at $\omega \in \Xx_t$ in the direction $\eta \in \Xx_t$ is given by
\begin{equation*}
\av{\partial_{\omega} \phi(X_T^{t,x,\omega}),\eta} 
= \phi'(X_T^{t,x,\omega}) \av{\partial_{\omega} X_T^{t,x,\omega},\eta}.
\end{equation*}
The Hadamard differentiability for the first order derivatives follows by a simple triangle inequality on the expectation. Indeed, for the derivative in $\omega$, we have for any sequences of positive numbers $(\varepsilon_n)_{n \in \NN}$ converging to zero and any sequence of paths $(\eta^{(n)})_{n \in \NN}$ converging to $\eta$ in $\Xx_t$: 
\begin{equation} \label{eq:triangle_inequality}
\begin{aligned}
\left| \ \EE \left[ \frac{\phi(X_T^{t,x,\omega+ \varepsilon_n \eta^{(n)}}) - \phi(X_T^{t,x,\omega}) }{\varepsilon_n} \right] - \EE[  \phi'(X_T^{t,x,\omega}) \av{\partial_{\omega} X_T^{t,x,\omega},\eta} ] \ \right| \\ 
\leq \EE \left[ \left| \frac{\phi(X_T^{t,x,\omega+ \varepsilon_n \eta^{(n)}}) - \phi(X_T^{t,x,\omega}) }{\varepsilon_n} - \phi'(X_T^{t,x,\omega}) \av{\partial_{\omega} X_T^{t,x,\omega},\eta} \right| \ \right],  
\end{aligned}
\end{equation}
and the right hand side of \eqref{eq:triangle_inequality} goes to $0$ as $n$ goes to $+\infty$ due to the Hadamard differentiability of $\omega \mapsto \phi(X_T^{t,x,\omega})$ for the $L^1(\Omega)$-norm. As announced previously, it is now clear that the map $\eta \mapsto \EE[\phi'(X_T^{t,x,\omega}) \av{\partial_{\omega} X_T^{t,x,\omega},\eta}]$ is linear, due to the linearity of $\eta \mapsto \av{\partial_{\omega} X_T^{t,x,\omega},\eta}$. 

To obtain the second derivative of $u$ with respect to $\omega$, we apply Lemma \ref{lm:hadamard_differentiability} with $f = \phi'$, $p = \kappa_{\phi}$, $r = 2$, and $q = 4 \kappa_{\phi} \vee 8$, combined with Lemma \ref{lm:gateaux_derivative_X}, showing that $\omega \mapsto \phi'(X_T^{t,x,\omega})$ is Hadamard-differentiable in $L^2(\Omega)$ and the chain rule applies from Proposition \ref{prop:chain_rule}. Moreover, Lemma \ref{lm:gateaux_derivative_X} also shows that the function $\omega \mapsto \av{\partial_{\omega} X_T^{t,x,\omega},\eta}$ is G\^{a}teaux-differentiable in $L^2(\Omega)$, and again the chain rule applies. This enables us to obtain the second order derivative of $u$ with respect to $\omega$. We can then apply the product rule from Lemma \ref{lm:product_rule} to show that $\omega \mapsto \phi'(X_T^{t,x,\omega}) \av{\partial_{\omega} X_T^{t,x,\omega},\eta}$ is G\^{a}teaux-differentiable in $L^1(\Omega)$, and the result follows.
The derivative $\p_x u$ is also shown to be Hadamard-differentiable in $\omega$ with similar arguments.

\paragraph{2. Continuity of the derivatives of $u$ under the distance $d_{\bar{\Lambda}}$.}

As explained previously, to show that a function is continuous under $d_{\bar{\Lambda}}$, it suffices to demonstrate its continuity with respect to each of the three variables $t$, $x$, and $\omega$. We will begin by showing that the mappings $(t,x,\omega) \mapsto X_T^{t,x,\omega}$, $(t,x,\omega) \mapsto \av{\p_{\omega} X_T^{t,x,\omega}, \eta}$, and $(t,x,\omega) \mapsto \av{\p^2_{\omega \omega} X_T^{t,x,\omega}, (\eta,\zeta)}$ are continuous in $L^p(\Omega)$ with respect to $t$ and $\omega$ for all $p \geq 1$ and $(\eta,\zeta) \in \Xx_t^2$. The continuity with respect to $x$ is straightforward since $X_T^{t,x,\omega} - X_T^{t,x',\omega} = x - x'$ and $\f{X_T^{t,x,\omega}}{\eta}$ and $\ff{X_T^{t,x,\omega}}{\eta}{\zeta}$ does not depend on $x$.

Let us focus first on time continuity: fix $(x,\omega) \in \RR \times \Xx_t$ and $(u,t) \in [0,T]^2$. Assume by symmetry that $u \leq t$. Then, by the definition of the stochastic flow,
\begin{equation} \label{eq:time_continuity_x}
\begin{aligned}
\| X_T^{u,x,\omega} - X_T^{t,x,\omega} \|_{L^p(\Omega)} \leq & \ \bigg\| \int_t^T ( b(r,V_r^{u,\omega}) - b(r,V_r^{t,\omega})) dr + \int_t^T ( \sigma(r,V_r^{u,\omega}) - \sigma(r,V_r^{t,\omega})) d B_r \bigg\|_{L^p(\Omega)}\\ & + \ \bigg\| \int_u^t b(r,V_r^{u,\omega}) d r + \int_u^t \sigma(r,V_r^{u,\omega}) d B_r \bigg\|_{L^p(\Omega)}.
\end{aligned}
\end{equation}
The second term on the right-hand side of \eqref{eq:time_continuity_x} is easily bounded by a constant multiplied by $t-u + \sqrt{t-u}$, using respectively Minkowski's integral inequality for the Lebesgue integral, BDG inequality followed by Minkoswki's integral inequality if $p \geq 2$ or by Jensen inequality if $p \in [1,2)$ for the stochastic integral, combined in both cases with the sub-exponential growth of $b$ and $\sigma$ and Lemma \ref{lm:controle_des_moments}. To handle the first term in the RHS of \eqref{eq:time_continuity_x}, we may use the same inequalities, and the locally Lipschitz nature of $b$ and $\sigma$ associated with a Cauchy-Schwarz inequality, the sub-exponential growth of $b$ and $\sigma$ and Lemma \ref{lm:controle_des_moments} reduces the control of the latter to that of $\| V_r^{u,\omega} - V_r^{t,\omega} \|_{L^q(\Omega)}$ for $q \in \{p, 2p\}$. Applying BDG inequality (with constant $c_q$) to the local martingale $(\int_u^z  K(r-\ell) d W_\ell, z \in [u,t])$, we have
\begin{equation*} 
\| V_r^{u,\omega} - V_r^{t,\omega} \|_{L^q(\Omega)} 
\leq c_{q}^{\frac{1}{q}} \bigg(\int_u^t K^2(r-\ell) d\ell \bigg)^{\frac{1}{2}}. 
\end{equation*}
Using Hypothesis \ref{hyp:kernel}, it follows that
\begin{equation*}
\| V_r^{u,\omega} - V_r^{t,\omega} \|_{L^q(\Omega)} \leqs c_{q}^{\frac{1}{q}} \bigg(\int_u^t (r-\ell)^{2H-1} d\ell \bigg)^{\frac{1}{2}} = c_{q}^{\frac{1}{q}} \sqrt{\frac{(r-u)^{2H} - (r-t)^{2H}}{2H}},
\end{equation*}
which tends to zero as $u \to t$, for fixed $r>u$, by the continuity of $z \mapsto z^{2H}$, and is uniformly bounded by $c_{q}^{\frac{1}{q}} \frac{T^H}{\sqrt{H}} \in L^1([0,T])$. Hence, $\| V_r^{u,\omega} - V_r^{t,\omega} \|_{L^q(\Omega)}$ tends to zero as $u \to t$ by the dominated convergence theorem for Lebesgue integration.

The continuity in $\omega$ is simpler to demonstrate. Fixing $t \in [0,T]$ and $x \in \RR$, Minkowski's integral inequality and BDG inequality followed by Minkoswki or Jensen inequality yield, for any $\omega, \tau \in \Xx_t$:
\begin{align*}
\| X_T^{t,x,\omega} - X_T^{t,x,\tau} \|_{L^p(\Omega)} \leq & \ \bigg\| \int_t^T ( b(r,V_r^{t,\omega}) - b(r,V_r^{t,\tau})) dr + \int_t^T ( \sigma(r,V_r^{t,\omega}) - \sigma(r,V_r^{t,\tau})) d B_r \bigg\|_{L^p(\Omega)} \\
\leq & \int_t^T \!\! \| b(r,V_r^{t,\omega}) - b(r,V_r^{t,\tau}) \|_{L^p(\Omega)} dr + c_p^{\frac{1}{p}} \left( \int_t^T \!\! \| (\sigma(r,V_r^{t,\omega}) - \sigma(r,V_r^{t,\tau})^2) \|_{L^{\frac{p}{2} \wedge 1}(\Omega)} \right)^{\frac{1}{2}}.
\end{align*}
Again, the locally Lipschitz nature of $b$ and $\sigma$ and the Cauchy-Schwarz inequality reduce the problem to the control of $\| V_r^{t,\omega} - V_r^{t,\tau} \|_{L^q(\Omega)}$ for $q \in \{p, 2p\}$. 
But since $V_r^{t,\omega} - V_r^{t,\tau} = \omega_r - \tau_r$, we have $ \| V_r^{t,\omega} - V_r^{t,\tau} \|_{L^q(\Omega)} \leq \| \omega - \tau \|_{[0,T]}$. It follows that $\| X_T^{t,x,\omega} - X_T^{t,x,\tau} \|_{L^p(\Omega)}$ tends to zero, as $\| \omega - \tau \|_{[0,T]}$ approaches zero. 
The proofs of time and $\omega$ continuity for the mappings $(t,x,\omega) \mapsto \av{\p_{\omega} X_T^{t,x,\omega}, \eta}$ and $(t,x,\omega) \mapsto \av{\p^2_{\omega \omega} X_T^{t,x,\omega}, (\eta,\zeta)}$ follow from the same arguments since these mappings share a structure very similar to $(t,x,\omega) \mapsto X_T^{t,x,\omega}$. 
The main difference is that one must bound the occurrences of $\eta_r$ and $\zeta_r$ by their supremum in time (which is finite by assumption) in the time integrals arising from the standard inequalities.

The continuity under $d_{\bar{\Lambda}}$ of all the derivatives of $u$ is thereby a consequence of the lemma below, whose proof is given in Appendix \ref{sec:proof_extended_continuity}:
\begin{lemma} \label{lm:extended_continuity} Let $f \in C^1(\RR,\RR)$ such that there exists $p \geq 1$ satisfying $|f(x)| + |f'(x)| \leqs 1 + |x|^p$ for all $x \in \RR$. For $q \geq 2p$, consider the mapping $F: L^q(\Omega) \to L^1(\Omega)$ defined by $F(X)(\omega) = f(X(\omega))$. 
Then $F$ is continuous from $L^q(\Omega)$ to $L^1(\Omega)$.
\end{lemma}
One may apply Lemma \ref{lm:extended_continuity} with $f = \phi', \phi''$, and since the composition of two continuous functions is continuous, it follows that $(t,x,\omega) \mapsto \phi^{(i)}(X_T^{t,x,\omega})$ is a continuous map from $(\bar{\Lambda},d_{\bar{\Lambda}})$ to $(L^1(\Omega),\| \cdot \|_{L^1(\Omega)})$ for $i = 1,2$, so the continuity of $\p_x u$, $\p_{xx} u$, $\p_{\omega} u$, $\p_{\omega} (\p_x u)$ and $\p_{\omega \omega}^2 u$ follows from their probabilistic representation \eqref{eq:probabilistic_representation} shown in the first part of the proof, the fact that a product of continuous functions is continuous and the triangle inequality.

\paragraph{3. Growth control of the derivatives.}
Observe that the growth control of $\p_x u$ and $\p_{xx} u$ follows directly from Lemma \ref{lm:controle_des_moments} since $\phi$ has at most polynomial growth. In addition, for any $\eta \in \Xx_t$,  we have by Cauchy-Schwarz inequality:
\begin{align*}
  |\f{u(t,x,\omega)}{\eta}| \leq \EE[ \phi'(X_T^{t,x,\omega})^2 ]^{\frac{1}{2}}  \  \EE[ \f{ X_T^{t,x,\omega}}{\eta}^2 ]^{\frac{1}{2}}.
\end{align*}
The first term in the right hand-side of the inequality above is bounded by some constant multiplied by $|x|^q + e^{q \| \omega \|_T}$ for some $q \geq 1$ thanks to \ref{hyp:regularity} and Lemma \ref{lm:controle_des_moments}, while the second term can be controlled using Jensen inequality and the Itô isometry:
\begin{align*}
  \EE[ \f{X_T^{t,x,\omega}}{\eta}^2 ] \leq & 2T \int_t^T \EE[\p_{x} b (r,V_r^{t,\omega})^2] \eta_r^2 dr + 2 \int_t^T \EE[\p_{x} \sigma (r,V_r^{t,\omega})^2] \eta_r^2 dr \\
  & \leqs \| \eta \ind_{[0,T]} \|_{[0,T]} (1 + e^{q' \| \omega \|_{[0,T]}}),
\end{align*}
for some $q' \geq 1$, using \ref{hyp:regularity} and Lemma \ref{lm:controle_des_moments}. 
The desired control of the first path derivative of $u$ follows afterwards, and we may get upper-bound for the second path derivative and the cross derivative using similar arguments.  
\end{proof}

\subsection{Regularity of $\omega\mapsto  \EE[\phi(X_T^{t,x,\omega})]$ -- the case of singular directions for $\p_\omega$} \label{sec:regularity_singular_directions}

In order to deal with directions that may be singular at the origin, we approximate the latter by a truncation.
Let $Q_{(1)},Q_{(2)} \in L^2([0,T],\RR_+)$ such that $Q$ is continuous on $(0,T]$.
For $t \in [0,T]$ and $Q \in \{ Q_{(1)}, Q_{(2)} \}$,  we set for all $u \in [t,T], \ Q^t(u) \coloneq Q(u-t)$ so that $Q^t \in L^2([t,T],\RR_+)$ and we define the truncated directions
\begin{equation*}
Q^{t,\delta}(s) = Q^t(s \vee (t+\delta)) 
= 
\left\{
\begin{array}{ll}
Q(s-t) & \mbox{if } s \geq t+\delta, \\
Q(\delta) & \mbox{if } s \in [0,t+\delta].
\end{array}
\right.
\end{equation*}
We may now define the path derivatives of a function $\func{u}{\bar{\Lambda}}{\RR}$ by the following limits, whenever they exist:
\begin{equation}\label{eq:spatial_derivative}
  \begin{aligned}
  \f{u(t,x,\omega)}{Q_{(1)}^t} 
  &=  \lim_{\delta \to 0} \f{u(t,x,\omega)}{Q_{(1)}^{t,\delta}}, \\
  \ff{u(t,\omega)}{Q_{(1)}^{t}}{Q_{(2)}^{t}} 
  &= \lim_{\delta \to 0} \ff{u(t,\omega)}{Q_{(1)}^{t,\delta}}{Q_{(2)}^{t,\delta}}.
  \end{aligned} 
\end{equation}
To ensure that these limits exist and that the functional It\^o formula and the path-dependent PDE hold, we need another set of regularity assumptions on $u$. The following conditions ensure that the path derivatives of $u$ are negligible when taken at directions supported near the time singularity.
\begin{definition}\label{def:Cplusalpha}
We say that $u\in C^{0,2,2}_{+,\alpha}(\Lambda)$, with $\alpha\in(0,1)$, if there exists a continuous extension of $u\in C^{2,2}_{+}(\overline\Lambda)$,
a constant $q > 0$ and a bounded modulus of continuity function~$\varrho$ such that, for any $x \in \mathbb{R}$, for any~$0\le t<T, \,0<\delta\le T-t$, and~$\eta,\zeta \in\Xx_t$ with supports contained in~$[t,t+\delta]$,
\begin{enumerate}
\item[(i)] for any~${\omega}\in\overline\Xx$ such that~${\omega}\one{[t,T]}\in\Xx_t$,
\begin{align*}
\abs{\langle \partial_{{\omega}} u(t,x,{\omega}),\eta\rangle} 
&\leq ( 1 + |x|^q + e^{ q \| \omega \|_{[0,T]}})  \norm{\eta}_{[0,T]} \delta^\alpha,\\
\abs{\langle \partial^2_{{\omega}{\omega}} u(t,x,{\omega}),(\eta,\zeta)\rangle} 
&\leq ( 1 + |x|^q + e^{ q \| \omega \|_{[0,T]}})  \norm{\eta}_{[0,T]} \norm{\zeta}_{[0,T]} \delta^{2\alpha}.
\end{align*}
\item[(ii)] for any other~${\omega}'\in\overline\Xx$ such that~${\omega}'\one{[t,T]}\in\Xx_t$,
\end{enumerate}
\small{
\begin{align*}
\abs{\langle \partial_{{\omega}} u(t,x,{\omega})-\partial_{{\omega}} u(t,x,{\omega}'),\eta\rangle} 
&\le ( 1 + |x|^q + e^{ q \| \omega \|_{[0,T]}} + e^{ q \| \omega' \|_{[0,T]}})  \norm{\eta}_{[0,T]} \varrho(\norm{{\omega}-{\omega}'}_{[0,T]}) \delta^\alpha,\\
\abs{\langle \partial^2_{{\omega}{\omega}} u(t,x,{\omega})-\partial_{{\omega}{\omega}} u(t,x,{\omega}'),(\eta,\zeta)\rangle} 
&\le ( 1 + |x|^q + e^{ q \| \omega \|_{[0,T]}} +  e^{ q \| \omega' \|_{[0,T]}})  \norm{\eta}_{[0,T]} \norm{\zeta}_{[0,T]} \varrho(\norm{{\omega}-{\omega}'}_{[0,T]})\ \delta^{2\alpha}.
\end{align*}
}
\begin{enumerate}
\item[(iii)] For any~${\omega}\in\overline\Xx$, $t\mapsto\langle \partial_{{\omega}} u(t,x,{\omega}),\eta\rangle$ and $t\mapsto\langle \partial^2_{{\omega}{\omega}} u(t,x,{\omega}),(\eta,\zeta)\rangle$ are continuous.
\end{enumerate}
\end{definition}

We will now give a representation of the derivatives \eqref{eq:spatial_derivative} taken at singular directions in our application case $u(t,x,\omega) = \EE[\phi(X_T^{t,x,\omega})]$. Note that, in this case, the direction $Q^t$ may not belong to $L^4([t,T])$. Consequently, we have to overcome the problem of the ill-defined stochastic integral $\int_t^T Q^2(s) dB_s$ appearing in the representation of the second order path derivative for continuous directions in Proposition \ref{prop:representation_first_order}. Nonetheless, an approach based on Malliavin calculus allows a representation for the second path derivative in this case, with the counterpart that is slightly more intricate compared to the one in Proposition \ref{prop:representation_first_order}. The following lemma gives the probabilistic representation for the derivatives of the function $u:(t,x,\omega) \mapsto \EE[\phi(\flowx_T)]$ taken at singular directions.

\begin{lemma} \label{lm:representation_first_order_singular}
  We assume that \ref{hyp:regularity} and \ref{hyp:kernel} hold. Let $(Q_{(1)},Q_{(2)}) \in L^2([0,T],\RR_+)\times L^2([0,T],\RR_+)$, such that $Q_{(i)}$ is continuous on $(0,T]$ and $\lim_{\delta \to 0} \delta Q_{(i)}^2(\delta) = \lim_{\delta \to 0} \int_0^\delta Q_{(i)}^2(r) dr = 0$ for $i=1,2$. Then for any $t\in[0,T]$ the conclusions of Proposition \ref{prop:representation_first_order} still hold when $(\eta, \zeta)\in \mathcal{X}_t^2$ is replaced by $(Q^t_{(1)},Q^t_{(2)}) \in L^2([t,T],\RR_+)^2$ with the following modification:
\begin{align*}
& \langle \partial^2_{\omega^2} u (t, x, \omega) , (Q^t_{(1)}, Q^t_{(2)}) \rangle \\
& =   \mathbb{E} [\phi'' (X_T^{t, x, \omega}) 
\langle \partial_{\omega} X_T^{t, x,\omega}, Q^t_{(1)} \rangle 
\langle \partial_{\omega} X_T^{t, x,  \omega}, Q^t_{(2)} \rangle ] \\ 
& \quad + \mathbb{E} \left[ \phi' (X_T^{t, x, \omega})
  \int_t^T \p_{xx}^2 b  (s,V_s^{t, \omega}) (Q_{(1)} Q_{(2)}) (t - s) d s \right]\\
& \quad+ \mathbb{E} \left[ \int_t^T \phi'' (X_T^{t, x, \omega}) (\sigma
  \p_{xx}^2  \sigma ) (s,V_s^{t, \omega}) (Q_{(1)} Q_{(2)}) (t - s) d s \right] \\
& \quad+ \rho \mathbb{E} \left[ \int_t^T \phi'' (X_T^{t, x, \omega}) \p_{xx}^2  \sigma (s,V_s^{t, \omega}) \left( \int_s^T \p_{x} b (\ell,V_{\ell}^{t, \omega}) K(\ell-s) d\ell \right) (Q_{(1)} Q_{(2)}) (t - s) d s \right] \\
&\quad + \rho \mathbb{E} \left[ \int_t^T \phi'' (X_T^{t, x, \omega}) \p_{xx}^2  \sigma (s,V_s^{t, \omega}) \left( \int_s^T \p_{x} \sigma (\ell,V_{\ell}^{t, \omega}) K(\ell-s) dB_{\ell} \right) (Q_{(1)} Q_{(2)}) (t - s) d s \right].
\end{align*}
\end{lemma}

\begin{proof}
The core of the proof is a straightforward adaptation of the proof of \citet[Proposition 2.23]{bonesini2023rough}, with the slight modification that $(K,K)$ is replaced by $(Q_{(1)}, Q_{(2)})$ and that we do not necessary have $b = - \frac{1}{2} \sigma^2$. This leads to
\begin{align}\label{eq:representation_second_deriv_priori}
\langle \partial^2_{\omega} u (t, x, \omega), (Q^t_{(1)}, Q^t_{(2)}) \rangle = &
  \mathbb{E} [\phi'' (X_T^{t, x, \omega}) \langle \partial_{\omega} X_T^{t, x,
  \omega}, Q^t_{(1)} \rangle  \langle \partial_{\omega} X_T^{t, x,
  \omega}, Q^t_{(2)} \rangle]\nonumber \\\nonumber & + \mathbb{E} \left[ \phi' (X_T^{t, x, \omega})
  \int_t^T \p_{xx}^2 b  (s, V_s^{t, \omega}) (Q_{(1)} Q_{(2)}) (t - s) d s \right]\\ 
  & + \mathbb{E} \left[ \int_t^T \mall{s}{\phi'(\itgr_T^{t,x,\omega})} \ \p_{xx}^2  \sigma  (s, V_s^{t, \omega})  (Q_{(1)} Q_{(2)}) (s - t) d s
  \right], 
\end{align}
where for $F \in \mathbb{D}^{1,2}(W,\overline{W})$, we recall that $\mall{s}{F} \coloneq \rho \mallw_s F + \brho \mallbw_s F$ with $\mallw$ and $\mallbw$ the Malliavin derivatives with respect to $W$ and $\overline{W}$. This adaptation is legitimate since our hypothesis \ref{hyp:regularity} contains Assumption 2.19 from \cite{bonesini2023rough}, while for $i = 1,2$, we have
\begin{align*}
\bigg(\int_t^{t+\delta} Q_{(i)}(s-t)^2 ds\bigg)^\frac{p}{2} \xrightarrow[\delta \to 0]{} & \ 0, \quad p \geq 2, \\
\int_t^{t+\delta} \left( Q_{(i)}(\delta) - Q_{(i)}(s-t) \right)^2 ds \xrightarrow[\delta \to 0]{} & \ 0 \\
\text{ and}\quad \int_t^{t+\delta} \left( Q_{(i)}(\delta)^2 - Q_{(i)}(s-t)^2 \right) ds \xrightarrow[\delta \to 0]{},&
\end{align*}
using that $\lim_{\delta \to 0} \delta Q_{(i)}^2(\delta) = \lim_{\delta \to 0} \int_0^\delta Q_{(i)}^2(r) dr = 0$.

Coming back to the equality in \eqref{eq:representation_second_deriv_priori}, since $X_T^{t,x,\omega}$ is explicit, we can compute the Malliavin derivative with the standard product and chain rules:
\begin{align*}
\mall{s}{\phi'(\itgr_T^{t,x,\omega})} 
& = \rho \mallw_s \phi' (X_T^{t, x, \omega}) + \brho   \mallbw_s\phi' (X_T^{t, x, \omega})\\
& = \phi'' (X_T^{t, x, \omega}) \ \left(\rho \mallw_s X_T^{t, x, \omega} +
\brho   \mallbw_s X_T^{t, x, \omega} \right).
\end{align*}
Moreover, since $V_{\cdot}^{t,\omega}$ is independent from $\widehat{W}$,
\begin{align*}
 \brho  \mallbw_s X_T^{t, x, \omega} 
& =  \brho  \mallbw_s \left( x + \int_t^T b (\ell, V_{\ell}^{t, \omega}) d \ell +
\rho \int_t^T \sigma (\ell, V_{\ell}^{t, \omega}) d W_{\ell} + \brho  
\int_t^T \sigma (\ell, V_{\ell}^{t, \omega}) d \widehat{W}_{\ell} \right) = \brho^2  \sigma (s, V_s^{t, \omega})
\end{align*}
and
\begin{align*}
  \rho  \mallw_s X_T^{t, x, \omega} = & \rho  \mallw_s  \left( x + \int_t^T b (\ell, V_{\ell}^{t,
  \omega}) d \ell + \rho \int_t^T \sigma (\ell, V_{\ell}^{t, \omega}) d W_{\ell} +
  \brho   \int_t^T \sigma (\ell, V_{\ell}^{t, \omega}) d
  \widehat{W}_{\ell} \right).
\end{align*}
By linearity of the Malliavin derivative, and since $\mallw_s V_{\ell}^{t,
\omega} = 0$ if $\ell < s$,
\begin{align*}
\mallw_s \left( \int_t^T b (\ell, V_{\ell}^{t, \omega}) d \ell \right) =  \int_s^T
\mallw_s b (\ell, V_{\ell}^{t, \omega}) d \ell
= \int_s^T \p_{x} b  (\ell, V_{\ell}^{t, \omega}) (\mallw_s V_{\ell}^{t, \omega}) d \ell,
\end{align*}
while we have
\begin{align*}
  \mallw_s \left( \int_t^T \sigma (\ell, V_{\ell}^{t, \omega}) d \widehat{W}_{\ell}
  \right) = & \int_s^T \mallw_s \sigma (\ell, V_{\ell}^{t, \omega}) d
  \widehat{W}_{\ell}
\end{align*}
and 
\begin{align*}
  \mallw_s \left( \int_t^T \sigma (\ell, V_{\ell}^{t, \omega}) d W_{\ell} \right) = &
  \sigma (s, V_s^{t, \omega}) + \int_s^T \p_{x} \sigma  (\ell, V_{\ell}^{t, \omega}) (\mallw_s
  V_{\ell}^{t, \omega}) d W_{\ell}.
\end{align*}
Hence, using that $B_{\ell} = \rho W_{\ell} + \brho  \widehat{W}_{\ell}$, we deduce that
\begin{align*}
  \rho  \mallw_s  X_T^{t, x, \omega} = & \rho^2 \sigma (s, V_s^{t, \omega}) + \rho\int_s^T \p_{x} b 
  (\ell,V_{\ell}^{t, \omega}) (\mallw_s V_{\ell}^{t, \omega}) d \ell + \rho\int_s^T \p_{x} \sigma 
  (\ell,V_{\ell}^{t, \omega}) (\mallw_s V_{\ell}^{t, \omega}) d B_{\ell}
\end{align*}
Now, for $\ell \geq s$, we have $\mallw_s V_{\ell}^{t, \omega} = \mallw_s  \left( \int_t^{\ell} K
(\ell - u) d W_u \right) = K (\ell - s)$. Finally, since $\rho^2 + \brho^2 = 1$, we conclude that
\begin{equation} \label{eq:malliavin_derivative_phi}
    \mall{s}{\phi'(\itgr_T^{t,x,\omega})}
    = \phi'' (X_T^{t, x, \omega}) \left( \sigma (s, V_s^{t, \omega}) + \rho
    \int_s^T \p_{x} b  (\ell, V_{\ell}^{t, \omega}) K (\ell - s) d \ell + \rho \int_s^T
    \p_{x} \sigma  (\ell, V_{\ell}^{t, \omega}) K (\ell - s) d B_{\ell} \right)
\end{equation}
and the result follows. 
\end{proof}

Using the representation of the derivatives, we may now check that $u$ has the necessary additional regularity to handle singular directions. 
\begin{proposition} \label{prop:regularity_singular_directions}
  Assume that \ref{hyp:regularity} holds. Then $u \in C^{0,2,2}_{+,\frac{1}{2}}(\Lambda)$. 
\end{proposition}
\proof{We have $u \in C^{2,2}_{+}(\bar{\Lambda})$ by Proposition \ref{prop:representation_first_order}. Therefore, it is sufficient to check that the path derivatives of $u$ in the sense of \eqref{eq:spatial_derivative} exist and satisfy the conditions of Definition \ref{def:Cplusalpha}. The existence of the derivatives of $u$ for singular directions follows by Lemma \ref{lm:representation_first_order_singular}.
For the estimates of Definition \ref{def:Cplusalpha}, we refer to the proof of \citet[Proposition 2.23]{bonesini2023rough} since our situation is a straightforward adaptation of the latter (with the difference that we do not necessary have $b = -\frac{1}{2} \sigma^2$, which has no impact in the obtained bounds).}

\subsection{Path dependant Feynman-Kac formula and Functional Itô formula} \label{sec:application_to_integrated_model}

We are now in a position to state the two main tools on which we base our analysis of  the weak error, namely the path-dependent Feynman-Kac formula and the functional Itô formula for the integrated model \eqref{eq:integrated_model}. 

\begin{proposition}[Path dependant Feynman-Kac formula] \label{prop:FK_formula}
Assume that \ref{hyp:regularity} and \ref{hyp:kernel} hold. Then the time derivative of $u:(t,x,\omega) \rightarrow\EE[\phi(X_T^{t,x,\omega})]$ exists (we thereby note $u \in C^{1,2,2}_{+,\frac{1}{2}}({\Lambda})$) and $u$ is the unique classical solution of the PPDE:
\begin{align*}
  \p_t u(t,x,\omega) & + \p_x u(t,x,\omega) b(t,\omega_t) + \frac{1}{2} \p_{xx} u(t,x,\omega) \sigma^2(t,\omega_t) \\
  & + \frac{1}{2} \ff{u(t,x,\omega)}{K(\cdot-t)}{K(\cdot - t)}  + \rho \ \sigma(t,\omega_t) \f{(\p_x u)(t,x,\omega)}{K(\cdot - t)} = 0,
\end{align*}
with terminal condition $u(T,x,\omega) = \phi(x)$, for any $(t,x,\omega) \in \bar{\Lambda}$.
\end{proposition}

\begin{proof}
Let $\bo{W} = (W,\widehat{W})$. We rewrite $(X,\overline{X})$ as a the two-dimensional processes $(Z,\overline{Z})$, and their associated dual-time semimartingales $\Pi$ and $\overline{\Pi}$ by
\begin{align*}
 Z_t = & z_0 + \int_0^t \gamma(t,r,Z_r) dr + \int_0^t \mu(t,r,Z_r) d\bo{W}_r, \quad \overline{Z}_t = z_0 + \int_0^t \gamma(t,r,\overline{Z}_r) dr + \int_0^t \nu(t,r,\overline{Z}_r) d\bo{W}_r, \\
  \Pi^t_s = & z_0 + \int_0^t \gamma(t,s,Z_r) dr + \int_0^t \mu(t,s,Z_r) d\bo{W}_r, \quad \overline{\Pi}^t_s = z_0 + \int_0^t \gamma(t,s,\overline{Z}_r) dr + \int_0^t \nu(t,s,\overline{Z}_r) d\bo{W}_r,
\end{align*}
where $z_0 = (X_0,0)$, $\gamma(t,r,(z_1,z_2)) = (b(t,z_2),0)$ and
\begin{align*}
  \mu(t,r,(z_1,z_2)) = \begin{pmatrix} &\rho \sigma(t,z_2)  &\brho \sigma(t,z_2)  \\ &K(t-r) &0 \end{pmatrix}, \quad \nu(t,r,(z_1,z_2)) = \begin{pmatrix} &\rho \sigma(t,z_2)  &\brho \sigma(t,z_2) \\ & \ak(t-r) & 0 \end{pmatrix}.
\end{align*}
By construction, we have $Z_t = (X_t,V_t)$, $\overline{Z}_t = (\overline{X}_t,\overline{V}_t)$, $\Pi^t_s = (X_t,\Theta^t_s), \overline{\Pi}^t_s = (\overline{X}_t,\overline{\Theta}^t_s)$. 
In particular, the function $v$ on $[0,T] \times C^0([0,T],\RR)^2$,  defined by $v(t,(\omega_1,\omega_2)) = \EE[\phi(({Z_{T}})^{t,(\omega_1,\omega_2)})] = u(t,\omega_1,\omega_2)$,  belongs to the space $C^{0,2}_{+,\frac{1}{2}}(\Lambda)$ in the sense of \cite{bonesini2023rough} since we showed in Proposition \ref{prop:regularity_singular_directions} that $u \in C^{0,2,2}_{+,\frac{1}{2}}(\Lambda)$.  
Then the result follows from the application of  \citet[Proposition 2.14]{bonesini2023rough} to the function $v$. 
The assumptions required for this are satisfied  due to the moment control from Lemma \ref{lm:controle_des_moments} and \ref{hyp:kernel}.
\end{proof}

\begin{proposition}[Functional Itô formula] \label{prop:functional_ito}
Assume that \ref{hyp:regularity} and \ref{hyp:kernel} hold. The following functional Itô formula holds for the function $u:(t,x,\omega) \rightarrow \EE[\phi(X_T^{t,x,\omega})]$ and the semimartingale $(\bitgr_t,\btheta^t)$:
\begin{align}\label{eq:Ito}
\begin{aligned}
u(t,\bitgr_t,\btheta^t) & =  u(0,X_0,0) + \int_0^t \p_t u(s,\bitgr_s,\btheta^s) ds + \int_0^t \p_x u(s,\bitgr_s,\btheta^s) b(s,\bvol_s) ds \\ 
&\quad + \frac{1}{2} \int_0^t \p_{xx} u(s,\bitgr_s,\btheta^s) \sigma^2(s,\bvol_s) ds + \int_0^t {\rho \sigma(s,\bvol_s)}\av{ \p_{\omega}( \p_x u )(s,\bitgr_s,\btheta^s), \ak(\cdot - s)} ds \\ 
&\quad + \frac{1}{2} \int_0^t \av{\p_{\omega \omega} u(s,\bitgr_s,\btheta^s), (\ak(\cdot - s),\ak(\cdot - s))} ds \\
&\quad  + \int_0^t \p_x u(t,\bitgr_t,\btheta^t) \sigma(s,\bvol_s) dB_s + \int_0^t \av{\p_{\omega} u(s,\bitgr_s,\btheta^s),\ak(\cdot - s)} dW_s.
  \end{aligned}
  \end{align}
\end{proposition}

\begin{proof}
Using the same notation as in the previous proof, we apply \citet[Theorem 3.17]{viens2019} to $v$ along with the two-dimensional Volterra process $\overline{Z}$ and its associated dual-time semimartingale $\overline{\Pi}$. 
  Due to Proposition \ref{prop:regularity_singular_directions} and Proposition \ref{prop:FK_formula}, $v$ belongs to $C^{1,2}_{+,\frac{1}{2}}(\Lambda)$. It follows that $\PP$-a.s., for all $t \in [0,T]$,
  \begin{align*}
   v(t,\overline{Z} \oplus_t \overline{\Pi}^t) = & v(0,(X_0,0)) + \int_0^t \partial_t v(r,\overline{Z} \oplus_r \overline{\Pi}^r) dr + \frac{1}{2} \int_0^t \langle \partial^2_{\omega} v(r,\overline{Z} \oplus_r \overline{\Pi}^r), (\sigma^{r,\overline{Z}},\sigma^{r,\overline{Z}}) \rangle dr \\
    & + \int_0^t \langle \partial_{\omega} v(r,\overline{Z} \oplus_r \overline{\Pi}^r), b^{r,\overline{Z}} \rangle dr + \int_0^t \langle \partial_{\omega} v(r,\overline{Z} \oplus_r \overline{\Pi}^r), \sigma^{r,\overline{Z}} \rangle d\bo{W}_r.
  \end{align*}
By definition, $v(t,\overline{Z} \oplus_t \overline{\Pi}^t) = \EE[ \phi(Z_T^{t,\overline{Z} \oplus_t \overline{\Pi}^t})]$ where
\begin{equation*}
 Z_T^{t,\overline{Z} \oplus_t \overline{\Pi}^t} = (\overline{Z} \oplus_t \overline{\Pi}^t)_T + \int_t^T \gamma(t,r,Z_r^{t,\overline{Z} \oplus_t \overline{\Pi}^t}) dr + \int_t^T \mu(t,r,Z_r^{t,\overline{Z} \oplus_t \overline{\Pi}^t}) d \bo{W}_r.
\end{equation*}
Since $(\overline{Z} \oplus_t \overline{\Pi}^t)_T = \overline{\Pi}^t_T$, the processes $Z_T^{t,\overline{Z} \oplus_t \overline{\Pi}^t}$ and $Z_T^{t,\overline{\Pi}^t}$ both are solution of the same stochastic Volterra equation started with the same initial condition, and since this equation is explicit in view of the coefficients $\gamma$ and $\mu$, it admits strong uniqueness and the above processes are indistinguishable, leading to 
\begin{equation*}
  v(t,\overline{Z} \oplus_t \overline{\Pi}^t) = v(t, \overline{\Pi}^t).
\end{equation*}
The result follows afterwards by computing explicitly the path derivatives of $v$ against the directions $b^{r,\overline{Z}}$ and $\sigma^{r,\overline{Z}}$, e.g using \citet[Remark 2.12]{bonesini2023rough}. 
\end{proof}

\section{Proof of Theorem \ref{thm:weakcvg_integrated}} \label{sec:proof_thm_1}

We denote 
\begin{equation*}
\begin{aligned}
\Ee_T  \coloneq  \EE[\phi(\bitgr_T)] - \EE[\phi(\itgr_T)] 
\end{aligned}
\end{equation*}
the weak error associated with the approximation of the integrated Volterra process \eqref{eq:integrated_model}.  First, we make use of the representation $\EE[\phi(\itgr_T)] = \EE[u(0,x,0)]$ and $\EE[\phi(\bitgr_T)] = \EE[u(T,\bxi_T,\btheta^T)]$. Applying the functional Itô formula in Proposition \ref{prop:functional_ito} on the function $u \in C^{1,2,2}_{+,\frac{1}{2}}(\Lambda)$ and the semimartingale $(\bxi,\btheta)$ between the times $0$ and $T$ and taking the expectation, we obtain
\begin{equation} \label{eq:first_order_dev}
\begin{aligned}
\Ee_T 
& =  \EE \int_0^T \p_t u(t,\bxi_t,\btheta^t) dt  + \EE \int_0^T \Big\{ b(t,\bvol_t) \p_x u(t,\bxi_t,\btheta^t) + \frac{1}{2} \sigma^2(t,\bvol_t) \pp_{xx} u(t,\bxi_t,\btheta^t) \Big\} dt.\\ 
& \quad   + \EE \int_0^T \Big\{ \rho \ \sigma(t,\bvol_t) \f{(\p_x u(t,\bxi_t,\btheta^t))}{\ak^t} + \frac{1}{2} \ff{u(t,\bxi_t,\btheta^t)}{\ak^t}{\ak^t} \Big\} dt. 
\end{aligned}
\end{equation}
Then thanks to the path-dependent Feynman-Kac formula in Proposition \ref{prop:FK_formula}, we can rewrite the first integral in the right-hand side of \eqref{eq:first_order_dev} in terms of the path derivatives of $u$. All the derivatives that are not taken in the direction of the kernel cancel out with the ones in \eqref{eq:first_order_dev}, and we are left with
\begin{equation*}
\begin{aligned}
    \Ee_T = & \rho \EE \int_0^T  \sigma(t,\bvol_t) \f{(\p_x u(t,\bxi_t,\btheta^t))}{\Delta K^t} dt \\ & + \frac{1}{2} \EE \int_0^T \Big\{ \ff{u(t,\bxi_t,\btheta^t)}{\ak^t}{\ak^t} - \ff{u(t,\bxi_t,\btheta^t)}{K^t}{K^t} \Big\} dt \\ 
    = & \rho \EE \int_0^T  \sigma(t,\bvol_t) \f{(\p_x u(t,\bxi_t,\btheta^t))}{\Delta K^t} dt  + \frac{1}{2} \EE \int_0^T \ff{u(t,\bxi_t,\btheta^t)}{\Delta K^t}{\Sigma K^t} dt,
\end{aligned}
\end{equation*}
where we used the bilinearity and the symmetry of the map $\partial^2_{\omega}u(t,\bxi_t,\btheta^t)$ to factorise the last term.
We use now the representation of the derivatives from Proposition \ref{prop:representation_first_order} and Lemma \ref{lm:representation_first_order_singular} with $Q_{(1)} = \Delta K$ and $Q_{(2)} = \Sigma K$, both belonging to $L^2([0,T],\RR_+)$, being continuous on $(0,T]$ and satisfying $\lim_{\delta \to 0} \delta Q_{(i)}^2(\delta) = \lim_{\delta \to 0} \int_0^\delta Q_{(i)}^2(r) dr = 0$ for $i=1,2$ by \ref{hyp:kernel}. From this, we can decompose each of the integrals above in sum of integrals  that we  analyse separately. From the representation of $\p_\omega \p_x u$, we get 
\begin{align*}
&\EE \left[ \int_0^T  \sigma(t,\bvol_t) \f{(\p_x u(t,\bxi_t,\btheta^t))}{\Delta K^t} dt \right] \\
&=  \int_0^T \int_t^T  \EE \Big[ \sigma(t,\bvol_t) \ \EE \Big[\phi''(\Flowx) \ \p_{x} b (s,\Flowv) \ \Big| \ \Ff_t \big] \Big] \Delta K(s-t) ds dt  \\
&  \quad + \int_0^T \EE \bigg[ \sigma(t,\bvol_t) \EE \bigg[\phi''(\Flowx) \int_t^T \p_{x} \sigma (s,\Flowv) \Delta K(s-t) dB_s \ \Big| \ \Ff_t  \bigg] \bigg] dt :=   {F}_1  +  {F}_2.
\end{align*}
Furthermore, from the representation of $\p^2_{\omega\omega} u$ we get 
\begin{align*}
& \EE\left[ \int_0^T \ff{u(t,\bxi_t,\btheta^t)}{\Delta K^t}{\Sigma K^t} dt \right] \\
& = \int_0^T \EE \Bigg[ \EE \bigg[ \phi''(\Flowx) \left( \int_t^T \p_{x} b(s,\Flowv) \Delta K(s-t) ds \right) \left( \int_t^T \p_{x} b(s,\Flowv) \Sigma K(s-t) ds \right) \ \Big| \ \Ff_t \bigg] \Bigg] dt \\
& \quad + \int_0^T \EE \Bigg[ \EE \bigg[ \phi''(\Flowx) \left( \int_t^T \p_{x} b (s,\Flowv) \Delta K(s-t) ds \right) \left( \int_t^T \p_{x} \sigma (s,\Flowv) \Sigma K(s-t) dB_s \right) \ \Big| \ \Ff_t \bigg] \Bigg] dt \\
& \quad +  \int_0^T \EE \Bigg[ \EE \bigg[ \phi''(\Flowx) \left( \int_t^T \p_{x} b (s,\Flowv) \Sigma K(s-t) ds \right) \left( \int_t^T \p_{x} \sigma (s,\Flowv) \Delta K(s-t) dB_s \right) \ \Big| \ \Ff_t \bigg] \Bigg] dt \\
& \quad + \int_0^T \EE \Bigg[ \EE \bigg[ \phi''(\Flowx) \left( \int_t^T \p_{x} \sigma (s,\Flowv) \Sigma K(s-t) dB_s \right) \left( \int_t^T \p_{x} \sigma (s,\Flowv) \Delta K(s-t) dB_s \right) \ \Big| \ \Ff_t \bigg] \Bigg] dt \\
& \quad +  \int_0^T \int_t^T \EE \Big[ \EE \big[ \phi'(\Flowx) \p_{xx}^2 b (s,\Flowv) \ \big| \ \Ff_t \big] \Big] \Delta (K^2)(s-t)  ds dt \\
& \quad +  \int_0^T \int_t^T \EE \Big[ \EE \big[ \phi''(\Flowx) (\sigma \p_{xx}^2  \sigma )(s,\Flowv) \Delta (K^2)(s-t) ds \ \big| \ \Ff_t \big] \Big] dt \\
& \quad +  \rho \ \int_0^T \int_t^T \EE \Bigg[ \EE \bigg[ \phi''(\Flowx) \p_{xx}^2  \sigma (s,\Flowv) \left( \int_s^T \p_{x} b (\ell,V_{\ell}^{t,\btheta^t}) K(\ell-s) d\ell \right) \Delta (K^2)(s-t) ds \ \Big| \ \Ff_t \bigg] \Bigg] dt \\
& \quad +  \rho \   \int_0^T \int_t^T \EE \Bigg[ \EE \bigg[ \phi''(\Flowx) \p_{xx}^2  \sigma (s,\Flowv) \left( \int_s^T \p_{x} \sigma (\ell,V_{\ell}^{t,\btheta^t}) K(\ell-s) dB_{\ell} \right) \Delta (K^2)(s-t) ds \ \Big| \ \Ff_t \bigg] \Bigg] dt.\\
& :=  {F}_3 +    {F}_4 +   {F}_5 +   {H}_1 +  {H}_2 +   {H}_3 + \rho \   {H}_4 + \rho  \   {H}_5. 
\end{align*}
We thus have obtained a decomposition of the error in  10 different terms: 
\begin{equation}\label{eq:error_decompos}
\Ee_T = \rho  ( {F}_1 +    {F}_2) + \frac{1}{2} \left( {F}_3 +    {F}_4 +   {F}_5 \right) +   \frac{1}{2} \left(  {H}_1 +  {H}_2 +   {H}_3 \right) + \frac{\rho}{2} \left(  {H}_4 +  {H}_5 \right). 
\end{equation}
These different terms, although sharing some similarities, each have a distinct structure.
Thus, a specific analysis is required for each of them to separate the two different contributions. 
Furthermore, a detailed analysis of each term also highlights the influence of the correlation $\rho$ between the two Brownian motions $B$ and $W$ on the constant multiplying the resulting error term.  
In particular, we will show that the terms $ {F}_1,  {F}_2,  {F}_3,  {F}_4$, and $ {F}_5$ contribute to the error term of the type $\| K - \ak \|_{L^1([0,T])}$, whereas the terms $ {H}_2$, $ {H}_3$, $ {H}_4$, and $ {H}_5$ contribute to the error term of the type $\| (K^2) - (\ak^2) \|_{L^1([0,T])}$. 
The term $ {H}_1$, slightly more technical to analyse, leads to both types of contributions.

An important point to highlight is that even in the case where $\rho = \sigma = 0$, the term $ {H}_2$ persists, involving the occurrence of the difference of squares $\| (K^2) - (\ak^2) \|_{L^1([0,T])}$.
In fact, this contribution to the error vanishes only under the highly restrictive conditions where $\p^2_{xx} b$, $\p^2_{xx} \sigma$, and $\p_{x}\sigma$ are all zero, which corresponds to a linear drift and a constant diffusion coefficient in the definition of the integrated Volterra process $X$. 

\subsection{Contributions in $\|K-\ak\|_{L^1([0,T])}$}

\subsubsection*{Estimate of $ {F}_1$}
We begin by the estimation of the term $ {F}_1 = \int_0^T \int_t^T  \EE \Big[ \sigma(t,\bvol_t) \ \EE \Big[\phi''(\Flowx) \ \p_{x} b (s,\Flowv) \ \Big| \ \Ff_t \big] \Big] \Delta K(s-t) ds dt$. This term contains only one occurrence of $\Delta K$ in a time integral, so it suffices to find a direct upper bound for the expectation of the stochastic terms, which is achieved using standard inequalities and our hypotheses. We will nevertheless detail the way to obtain this upper bound in the analysis of $ {F}_1$, allowing us to go faster on this point in the analysis of the other error terms.
We apply the triangle, Cauchy-Schwarz and Jensen inequalities:
\begin{align*}
| {F}_1| & \leq \int_0^T \int_t^T \mathbb{E} \left[ \left| \sigma(t,\bvol_t) \EE \big[\phi''(\Flowx) \p_{x} b (s,\Flowv) \ \big| \ \Ff_t \big] \right| \right] |\Delta K| (s-t) ds dt \\
& \leq \int_0^T \int_t^T \mathbb{E} \left[ \sigma(t,\bvol_t)^2 \right]^{\frac{1}{2}} \left(  \mathbb{E} \left[ \left| \EE \big[\phi''(\Flowx) \p_{x} b (s,\Flowv) \ \big| \ \Ff_t \big] \right|^2 \right] \right)^{\frac{1}{2}} |\Delta K| (s-t) ds dt \\
& \leq \int_0^T \mathbb{E} \left[ \sigma(t,\bvol_t)^2 \right]^{\frac{1}{2}} \int_t^T   \left( \mathbb{E} \left[ \phi''(\itgr_T^{t,\bitgr_t,\btheta^t})^2 \p_{x} b (s,\vol_s^{t,\btheta^t})^2 \right] \right)^{\frac{1}{2}} |\Delta K| (s-t) ds dt.
\end{align*}
By \ref{hyp:regularity}, one has $\sigma(t,x) \leqs 1 + \exp(\nu_{\sigma} (x+t))$, $\p_{x} b (s,x) \leqs 1+ \exp(\nu_{b} (x+t))$ and $|\phi''(x)| \leqs 1 + |x|^{\kappa_\phi}$. It follows, since $\bvol_t = \vol_t^{t,\btheta^t}$,  that
\begin{align*}
\mathbb{E} \left[ \sigma(t,\bvol_t)^2 \right] \leqs  1 + e^{2 \nu_\sigma T} \sup_{t \in [0,T]} \EE[ \exp(2 \nu_\sigma \bvol_t )] \leqs 1 + e^{2 \nu_\sigma T} \sup_{t \in [0,T]} \sup_{s \in [t,T]} \EE[ \exp(2 \nu_\sigma \vol_s^{t,\btheta^t} )].
\end{align*}
Due to the moment control from Lemma \ref{lm:controle_des_moments}, we obtain
\begin{align*}
\left( \mathbb{E} \left[ \sigma(t,\bvol_t)^2 \right] \right)^\frac{1}{2} 
\leqs \left( 1 + e^{\overline{m}^{(2)}_{2 \nu_{\sigma}} \|K\|_{L^2([0,T])}^2 } \right)^\frac{1}{2} \leqs 1 + e^{\overline{m}^{(2)}_{2 \nu_{\sigma}} \|K\|_{L^2([0,T])}^2 }.
\end{align*}
For the second expectation, we apply Cauchy-Schwarz a second time, and using similar arguments, we deduce that
\begin{align*}
\left( \mathbb{E} \left[ \phi''(\itgr_T^{t,\bitgr_t,\btheta^t})^2 \p_{x} b (s,\vol_s^{t,\btheta^t})^2 \right] \right)^{\frac{1}{2}} \leqs & (1 + e^{\overline{m}^{(1)}_{4 \kappa_{\phi}} \|K\|_{L^2([0,T])}^2 })^{\frac{1}{4}} (1 + e^{\overline{m}^{(2)}_{4 \nu_{b}} \|K\|_{L^2([0,T])}^2 })^{\frac{1}{4}} \\
  & \leq (1 + e^{\overline{m}^{(1)}_{4 \kappa_{\phi}} \|K\|_{L^2([0,T])}^2 })(1 + e^{\overline{m}^{(2)}_{4 \nu_{b}} \|K\|_{L^2([0,T])}^2 }) \\
  & \leq 3 + 3 e^{ M_{\kappa_\phi, \nu_b} \|K\|_{L^2([0,T])}^2 } \leqs 1 + e^{ M_{\kappa_\phi, \nu_b} \|K\|_{L^2([0,T])}^2 }
\end{align*}
where $ M_{\kappa_\phi, \nu_b} \coloneqq \max\{ \overline{m}^{(1)}_{4 \kappa_{\phi}} \ ; \ \overline{m}^{(2)}_{4 \nu_{b}} \ ; \  \overline{m}^{(1)}_{4 \kappa_{\phi}} \overline{m}^{(2)}_{4 \nu_{b}} \}$. 
Therefore, from the previous estimation and the change of variable $u = s-t$, we have
\begin{align*}
| {F}_1| \leqs & (1 + e^{C_{ {F}_1} \|K\|_{L^2([0,T])}^2 }) \int_0^T \int_t^T |\Delta K| (s-t) ds dt \\
& = (1 + e^{ C_{ {F}_1} \|K\|_{L^2([0,T])}^2 }) \int_0^T \| \Delta K \|_{L^1([0,t])} dt
\end{align*}
for  the constant $C_{ {F}_1} \coloneqq \max\{ \overline{m}^{(2)}_{2 \nu_{\sigma}} \ ; \ M_{\kappa_\phi, \nu_b} \ ; \  \overline{m}^{(2)}_{2 \nu_{\sigma}} M_{\kappa_\phi, \nu_b} \} $ that does not depend on $K$. 

\subsubsection*{Estimate of $ {F}_2$}
We now turn to the second term $ {F}_2 = \int_0^T \EE \bigg[ \sigma(t,\bvol_t) \EE \bigg[\phi''(\Flowx) \int_t^T \p_{x} \sigma (s,\Flowv) \Delta K(s-t) dB_s \ \Big| \ \Ff_t  \bigg] \bigg] dt$. As in the term $F_1$, there is an occurrence of $\Delta K$, but it now takes place in a stochastic integral. A straightforward use of the Itô isometry would result in a time integral of $(\Delta K)^2$, which is not desirable since it is analogue to the strong rate of convergence. Instead, the use of Malliavin calculus allows the recovery of a time integral of $\Delta K$ itself. 
More precisely, we make use of the (conditional) Malliavin integration by part, which is legitimate since $\phi'' \in C^1(\RR,\RR)$ by \ref{hyp:regularity}. This leads to
\begin{equation}\label{eq:estimate_f2}
\begin{aligned}
{F}_2 & =  \int_0^T \EE \left[ \sigma(t,\bvol_t) \int_t^T \EE \left[ 
\mall{s}{\phi''(\Flowx)} \ \p_{x} \sigma (s,\Flowv) \ \big| \ \Ff_t \right] \Delta K(s-t) ds  \right] dt \\
& = \int_0^T   \EE 
\left[ \int_t^T \EE \left[ \mall{s}{\phi''(\Flowx)} \ \p_{x} \sigma (s,\Flowv) \sigma(t,\bvol_t) \ \Big| \ \Ff_t \right]  
\Delta K(s-t) ds \right] dt \\
& = \int_0^T \int_t^T \EE \left[ \mall{s}{\phi''(\Flowx)} \ \p_{x} \sigma (s,\Flowv) \sigma(t,\bvol_t) \right]  
\Delta K(s-t) ds dt.
\end{aligned}
\end{equation} 
We recall the representation \eqref{eq:malliavin_derivative_phi} of the Malliavin derivative:
\small{
\begin{equation*}
  \mall{s}{\phi''(\Flowx)}
    = \phi''' (\Flowx) \left(\sigma (s, \Flowv) + \rho
    \int_s^T \p_{x} b  (\ell,V_{\ell}^{t,\btheta^t}) K (\ell - s) d \ell + \rho \int_s^T
    \p_{x} \sigma  (\ell,V_{\ell}^{t,\btheta^t}) K (\ell - s) d B_{\ell} \right).
\end{equation*}
}
The process $(\p_{x} \sigma (\ell,V_{\ell}^{t,\btheta^t}) K (\ell - s))_{\ell \in [s,T]}$ is adapted to the filtration $(\mathcal{F}_{\ell})_{\ell \in [s,T]}$ and left-continuous hence predictable, since $\ell \mapsto K(\ell-s)$ is left-continuous on $[s,T]$ by \ref{hyp:kernel} and $\p_{x} \sigma (\cdot, v)$ is continuous for almost every $v \in \RR$ by \ref{hyp:regularity}. It follows that the process $(\int_s^t \p_{x} \sigma (\ell,V_{\ell}^{t,\btheta^t}) K (\ell - s) d B_{\ell})_{t \in [s,T]}$ is a $(\mathcal{F}_{t})_{t \in [s,T]}$-local martingale, allowing to apply BDG inequality. 
Hence, for any $p \geq 1$, we may apply Cauchy-Schwarz inequality, Minkoswki's integral inequality and BDG inequality with constant $c_p$ to get
\begin{align*}
& \| \mall{s}{\phi''(\itgr_T^{t,\bitgr_t,\btheta^t})} \|_{L^p(\Omega)} \\
& \leq \| \phi'''(\itgr_T^{t,\bitgr_t,\btheta^t}) \|_{L^{2p}(\Omega)} \| \sigma(s,\vol_s^{t,\btheta^t}) \|_{L^{2p}(\Omega)} \\
& \quad + \rho  \| \phi'''(\itgr_T^{t,\bitgr_t,\btheta^t}) \|_{L^{2p}(\Omega)}  \Bigg(  \int_s^T \| \p_{x} b (\ell,\vol_{\ell}^{t,\btheta^t}) \|_{L^{2p}(\Omega)} K(\ell-s) d\ell +  c_{2p}^{\frac{1}{2p}} \left(\int_s^T \| \p_{x} \sigma (\ell, \vol_{\ell}^{t,\btheta^t})^2 \|_{L^{p}(\Omega)} K^2(\ell-s) d\ell \right)^{\frac{1}{2}} \Bigg).
\end{align*}
Then exploiting the growth control of $\phi''', \sigma, \p_{x} b $ and $\p_{x} \sigma $ given by \ref{hyp:regularity} and the moment control from Lemma \ref{lm:controle_des_moments}, we derive that for every $p \geq 1$,
\begin{equation}  \label{eq:estimate_mall_phi}
  \sup_{t \in [0,T]} \sup_{s \in [t,T]} \| \mall{s}{\phi''(\itgr_T^{t,\bitgr_t,\btheta^t})} \|_{L^p(\Omega)} \leqs (1 + e^{C_p^{\textit{Mall}} \|K\|_{L^2([0,T])}} ) (1 + \rho (\| K \|_{L^1([0,T])} +  \|K\|_{L^2([0,T])})),
\end{equation}
where the constant $C_p^{\textit{Mall}} >0 $ does not depend on $K$. We get back to \eqref{eq:estimate_f2}, apply Hölder inequality with $p,q,r \geq 1$ such that $\frac{1}{p} + \frac{1}{q} + \frac{1}{r} = 1$, use the estimate \eqref{eq:estimate_mall_phi} along with the growth and moment control to handle the terms $\| \p_{x} \sigma (s,V_s^{t,\btheta^t}) \|_{L^q(\Omega)}$ and $\| \sigma(t,V_t^{t,\btheta^t}) \|_{L^r(\Omega)}$. We obtain:
\begin{align*}
  | {F}_2| \leqs & (1 + e^{C_{ {F}_2} \|K\|_{L^2([0,T])}} ) (1 + \rho (\| K \|_{L^1([0,T])} +  \|K\|_{L^2([0,T])}))  \int_0^T \| \Delta K \|_{L^1([0,t])} dt
\end{align*}
with $C_{ {F}_2} >0$ not depending on $K$. 

Notice that the inequality $\rho \leq \indic{\rho > 0}$ holds for $\rho \in [0,1]$, and using \ref{hyp:kernel} we have
\begin{align*}
  \| K \|_{L^1([0,T])} = \int_0^T K(t) dt \leqs \int_0^T t^{H-\frac{1}{2}} dt = \frac{1}{2H+1} T^{H+\frac{1}{2}} \leqs T\vee 1.
\end{align*}
So that we can write
\begin{align*}
  | {F}_2| \leqs & (1 + e^{C_{ {F}_2} \|K\|_{L^2([0,T])}} ) (1 + \indic{\rho > 0} \|K\|_{L^2([0,T])})  \int_0^T \| \Delta K \|_{L^1([0,t])} dt.
\end{align*}

\subsubsection*{Estimate of $ {F}_3$}
%$ {F}_3 = \int_0^T \int_t^T \int_t^T \EE \Big[ \ctE \big[ \phi''(\itgr_T) b'(s_1,\vol_{s_1}) b'(s_2,\vol_{s_2}) \big] \Big] (K^t_{s_1} - \ak^t_{s_1}) (K^t_{s_2} + \ak^t_{s_2}) ds_1 ds_2 dt$
We consider $F_3 =  \int_0^T \EE \bigg[ \phi''(\Flowx) \left( \int_t^T \p_{x} b(s,\Flowv) \Delta K(s-t) ds \right) \left( \int_t^T \p_{x} b(s,\Flowv) \Sigma K(s-t) ds \right) \bigg] dt$. The technique is very close to what is done for $ {F}_1$. One may observe that if we erase the random terms involved in the definition of ${F}_3$, the latter just consists of the integral of a product between the time integral of $\Delta K$ and the time integral of $\Sigma K$. It is therefore sufficient to find a direct upper-bound for this expectation. Toward this view, we use the Fubini-Lebesgue theorem to push the expectations inside the integral, leading to
\begin{align*}
F_3 = \int_0^T \int_t^T \int_t^T \EE \Big[ \phi''(\Flowx) \p_x b(s_1,\vol_{s_1}^{t,\btheta^t}) \p_x b(s_2,\vol_{s_2}^{t,\btheta^t}) \Big] \Delta K(s_1 - t) \Sigma K(s_2 - t) ds_1 ds_2 dt.
\end{align*}
We can now apply Hölder inequality with $p,q,r \geq 1$ such that $\frac{1}{p} + \frac{1}{q} + \frac{1}{r} = 1$ and use the growth and moment controls \ref{hyp:regularity} and Lemma \ref{lm:controle_des_moments}. We obtain straightforwardly that, for some constant $C_{ {F}_3} >0$ not depending on $K$:
\begin{align*}
  | {F}_3| \leqs & (1 + e^{C_{ {F}_3} \|K\|_{L^2([0,T])}^2 }) \int_0^T \int_t^T \int_t^T |\Delta K(s_1-t)| \Sigma K(s_2-t) ds_1 ds_2 dt \\
  & \leq (1 + e^{C_{ {F}_3} \|K\|_{L^2([0,T])}^2 }) \int_0^T \| \Delta K \|_{L^1([0,t])} dt \times  \| \Sigma K \|_{L^1([0,T])}.
\end{align*}
Using \ref{hyp:kernel}, we have
\begin{align*}
  \| \Sigma K \|_{L^1([0,T])} = & \int_0^T (K(s) + \ak(s)) ds
   \leq (1+c_{K,\ak}) \int_0^T K(s) ds \leqs \int_0^T s^{H-\frac{1}{2}} ds = \frac{1}{2H+1} T^{H+\frac{1}{2}}.
\end{align*}
Then, since $H \in (0,\frac{1}{2})$, we have
\begin{equation} \label{eq:control_sigma_K_L1}
  \| \Sigma K \|_{L^1([0,T])} \leqs T\vee 1,
\end{equation}
and therefore,
\begin{align*}
  | {F}_3| \leqs & (1 + e^{C_{ {F}_3} \|K\|_{L^2([0,T])}^2 })  \int_0^T \| \Delta K \|_{L^1([0,t])} dt.
\end{align*}

\subsubsection*{Estimate of $ {F}_4$ }

We recall that $ {F}_4 = \int_0^T \EE \Big[ \phi''(\Flowx) \left( \int_t^T \p_x b(s,\Flowv) \Delta K(s-t) ds \right) \left( \int_t^T \p_x \sigma(s,\Flowv) \Sigma K(s-t) dB_s \right) \Big] dt$. 
This expression is quite standard to deal with as it already contains a time integral of $\Delta K$. Nonetheless, we will separate the cases where $\rho > 0$ and $\rho = 0$, as we can achieve a better constant in the latter case by using Malliavin calculus. 
When $\rho > 0$, we apply Hölder inequality with $p=4,q=4,r=2$ so that $\frac{1}{p}+\frac{1}{q}+\frac{1}{r}=1$. We use Minkoswki integral inequality and BDG inequality\footnote{Note that the process $\left( \int_t^r \p_x\sigma(s,\vol_s) \Sigma K(s-t) dB_s \right)_{r \in [t,T]}$ is a $(\Ff_r)_{r \in [t,T]}$-local martingale because $s \mapsto \p_x\sigma(s,\vol_s) \Sigma K(s-t)$ is $(\Ff_s)_{s \in [t,T]}$-adapted and left-continuous on $[t,T]$, hence $(\Ff_s)_{s \in [t,T]}$-predictable. 
This justifies the application of the BDG inequality.} to obtain 
\begin{align*}
  | {F}_4| \leq \ c_1  \int_0^T \Bigg\{ \| \phi''(X_T^{t,\bitgr_t,\btheta^t}) \|_{L^4(\Omega)}
  &  \left( \int_t^T  \| \p_{x} b (V_s^{t,\btheta^t}) \|_{L^4(\Omega)} \Delta K(s-t) ds \right) \\ 
& \qquad \times \left( \int_t^T \| \p_{x} \sigma (V_s^{t,\btheta^t})^2 \|_{L^{1}(\Omega)} \Sigma K(s-t)^2 ds \right)^\frac{1}{2} \Bigg\} dt.
\end{align*}
Then using \ref{hyp:regularity}, Lemma \ref{lm:controle_des_moments} and \ref{hyp:kernel}, we deduce that
\begin{align*}
  | {F}_4| \leqs & \ \|K \|_{L^2([0,T])} \ (1 + e^{C'_{ {F}_4} \|K\|_{L^2([0,T])}^2 }) \int_0^T \| \Delta K \|_{L^1([0,t])} dt
\end{align*}
for some $C'_{ {F}_4} >0$ not depending on $K$.

When $\rho = 0$, we use the Malliavin integration by part, taking advantage of the fact that $\mall{s}{\p_x b(V_s)} = \rho \mallw_s \p_x b(V_s) = 0$  and that $\mall{s}{\phi''(X_T)} = \phi'''(X_T) \sigma(V_s)$ for all $s \in [t,T]$. We obtain
\begin{align*}
   {F}_4 = & \int_0^T \int_t^T \int_t^T \EE [ \phi'''(\itgr_T^{t,\bitgr_t,\btheta^t}) (\sigma \p_{x} \sigma ) (s,V_s^{t,\btheta^t}) \p_{x} b (r,V_r^{t,\btheta^t}) \sigma(s,V_s^{t,\btheta^t}) ] \Delta K(r-t) \Sigma K(s-t) dr ds dt.
\end{align*}
We may estimate the expectation of the stochastic processes using Hölder inequality, \ref{hyp:regularity} and Lemma \ref{lm:controle_des_moments} and \eqref{eq:control_sigma_K_L1} to get
\begin{align*}
  | {F}_4| \leqs & (1 + e^{C''_{ {F}_4} \|K\|_{L^2([0,T])}^2 }) \int_0^T \| \Delta K \|_{L^1([0,t])} dt 
\end{align*}
for some $C''_{ {F}_4} >0$ not depending on $K$.
Letting $C_{ {F}_4} = \max({C'_{ {F}_4},C''_{ {F}_4}})$, we conclude that for every $\rho \in [0,1]$, 
\begin{align*}
  | {F}_4| \leqs & (1 + \indic{\rho > 0} \|K \|_{L^2([0,T])}) (1 + e^{C_{ {F}_4} \|K\|_{L^2([0,T])}^2 }) \int_0^T \| \Delta K \|_{L^1([0,t])} dt.
\end{align*}

\subsubsection*{Estimate of $ {F}_5$}
Recall that $ {F}_5 = \int_0^T \EE \Big[  \phi''(\Flowx) \left( \int_t^T b'(s,\Flowv) \Sigma K(s-t) ds \right) \left( \int_t^T \sigma'(s,\Flowv) \Delta K(s-t) dB_s \right) \Big] dt$. For this expression, it is convenient to apply the Malliavin integration by part, since the occurrence of $\Delta K$ lies in the stochastic integral. Thus we write
\begin{equation*}
   {F}_5 =  \int_0^T \int_t^T \EE \left[ \mallext{s}{\phi''(\Flowx)  \int_t^T \p_{x} b (r,\Flowv) \  \Sigma K(r-t) dr }  \p_{x} \sigma (s,\Flowv) \right] \Delta K(s-t) ds dt.
\end{equation*}
By the product rule and the inversion of the Lebesgue integral with the Malliavin derivative, we have
\begin{align*}
  \mallext{s}{\phi''(\Flowx)  \int_t^T \p_{x} b (r,\vol_r^{t,\btheta^t}) \Sigma K(r-t) dr } = & \mall{s}{\phi''(\Flowx)} \left( \int_t^T \p_{x} b (r,\vol_r^{t,\btheta^t}) \Sigma K(r-t) dr \right) \\
  & + \phi''(\Flowx) \left( \int_t^T \!\!\! \mall{s}{\p_{x} b (r,\vol_r^{t,\btheta^t})} \Sigma K(r-t) dr \right)
\end{align*}
where $\mall{s}{\p_{x} b (r,\vol_r^{t,\btheta^t})} = \rho \p_{xx}^2 b (r,\vol_r^{t,\btheta^t}) \ \mallw_s \vol_r^{t,\btheta^t} = \rho \p_{xx}^2 b (r,\vol_r^{t,\btheta^t}) K(r-s) \indic{r > s}$. We can then write
\begin{align*}
   {F}_5 = &  \int_0^T \int_t^T \EE \Bigg[  \mall{s}{\phi''(\Flowx)} \left( \int_t^T \p_{x} b (r,\vol_r^{t,\btheta^t}) \Sigma K(r-t) dr \right)  \p_{x} \sigma (s,\vol_s^{t,\btheta^t}) \bigg] \Delta K(s-t) ds  dt \\
  & + \rho \int_0^T \int_t^T \EE \Bigg[  \phi''(\Flowx) \left( \int_s^T \p_{xx}^2 b (r,\vol_r^{t,\btheta^t}) K(r-s) \Sigma K(r-t) dr \right) \p_{x} \sigma (s,\Flowv) \bigg]  \Delta K(s-t) ds dt.
\end{align*}
We may apply triangle inequality, Hölder inequality with $p,q,r \geq 1$ such that $\frac{1}{p} + \frac{1}{q} + \frac{1}{r} = 1$ and Minkoswki integral inequality to derive \small{
\begin{align*}
& | {F}_5|\\ 
&\leq  \int_0^T \int_t^T \| \mall{s}{\phi''(\Flowx)} \|_{L^p(\Omega)}  \left( \int_t^T \| \p_{x} b (r,\vol_r^{t,\btheta^t}) \|_{L^q(\Omega)} \Sigma K(r-t) dr \right) \| \p_{x} \sigma (s,\Flowv) \|_{L^r(\Omega)} |\Delta K| (s-t) ds  dt \\
& \ \  + \rho \int_0^T\!\!\!\int_t^T \bigg\{ \| \phi''(\Flowx) \|_{L^p(\Omega)} \left( \int_s^T \| \p_{xx}^2 b (r,\vol_r^{t,\btheta^t}) \|_{L^q(\Omega)} K(r-s) \Sigma K(r-t) dr \right)\| \p_{x} \sigma (s,\Flowv)\|_{L^r(\Omega)} | \Delta K| (s-t) \bigg\} ds dt.
\end{align*}}
Then using \ref{hyp:regularity}, Lemma \ref{lm:controle_des_moments}, \eqref{eq:estimate_mall_phi} and \eqref{eq:control_sigma_K_L1}, there exists a constant $C_{ {F}_5} >0$ not depending on $K$ such that
\begin{align*}
  | {F}_5| \leqs & (1+ e^{C_{ {F}_5} \|K\|_{L^2([0,T])}^2 }) \int_0^T \int_t^T | \Delta K| (s-t) \bigg(1+ \rho \|K\|_{L^2([0,T])} + \rho \int_s^T K(r-s) \Sigma K(r-t) dr \bigg) ds dt.
\end{align*}
We know from \ref{hyp:kernel} that $K(r-s) \Sigma K(r-t) \leq (1+c_{K,\ak}) K(r-s) K(r-t) \leq (1+c_{K,\ak}) K^2(r-s)$ since $K$ is non-increasing and $s > t$. We can then write
\begin{align*}
  \int_s^T K(r-s) \Sigma K(r-t) dr \leq & (1+c_{K,\ak}) \int_s^T K^2(r-s) dr \leq (1+c_{K,\ak}) \int_0^T K^2(r) dr = (1+c_{K,\ak}) \|K\|_{L^2([0,T])}^2.
\end{align*}
Using again the inequality $\rho \leq \indic{\rho > 0}$, we obtain
\begin{align*}
  | {F}_5| \leqs & (1+ e^{C_{ {F}_5} \|K\|_{L^2([0,T])}^2 }) (1+ \indic{\rho > 0} (\|K\|_{L^2([0,T])} + \|K\|_{L^2([0,T])}^2)) \int_0^T \| \Delta K \|_{L^1([0,t])} dt.
\end{align*}

\subsection{Contributions in $\|K^2 - \ak^2 \|_{L^1([0,T])}$}

\subsubsection*{Estimate of $ {H}_1$}
The expression $  {H}_1 = \int_0^T \EE \Bigg[ \phi''(\Flowx) \left( \int_t^T \p_{x}\sigma (s,\Flowv) \Sigma K(s-t) dB_s \right) \left( \int_t^T \p_{x}\sigma(s,\Flowv) \Delta K(s-t) dB_s \right) \Bigg] dt$ is the trickiest to handle since it contains the product of two stochastic integrals, which are themselves multiplied by the unsuitable stochastic process $\phi''(X_T^{t,\bitgr_t,\btheta^t})$, which prohibits the use of Itô's isometry.

With the use of the Malliavin integration by part, 
\begin{align*}
   {H}_1 = & \int_0^T \int_t^T \EE \Bigg[  \mallext{s}{\phi''(\Flowx) \int_t^T \p_{x} \sigma (r,\vol_r^{t,\btheta^t}) \Sigma K(r-t) dB_r} \p_{x} \sigma (s,\Flowv) \Bigg] \Delta K(s-t) ds dt.
\end{align*}
Then,  the product rule and the identity
\begin{align*}
 &\mallext{s}{\int_t^T \p_{x} \sigma (r,\vol_r^{t,\btheta^t}) \Sigma K(r-t) dB_r} \\
 & = \p_{x} \sigma (s,V_s^{t,\btheta^t}) \Sigma K(s-t) + \int_s^T \mall{s}{\p_{x} \sigma (r,\vol_r^{t,\btheta^t})} \Sigma K(r-t) dB_r \\
 & = \p_{x} \sigma (s,V_s^{t,\btheta^t}) \Sigma K(s-t) + \rho  \int_s^T \p_{xx}^2  \sigma (r,\vol_r^{t,\btheta^t}) K(r-s) \Sigma K(r-t) dB_r 
\end{align*}
decompose $ {H}_1$ into the three following parts:
\begin{align*}
   {H}_1 
  = & \int_0^T \int_t^T \EE \Bigg[  \mall{s}{\phi''(\Flowx)} \p_{x} \sigma (s,\vol_s^{t,\btheta^t}) \int_t^T \p_{x} \sigma (r,\vol_r^{t,\btheta^t}) \Sigma K(r-t) dB_r  \Bigg] \Delta K(s-t) ds dt \\
  & + \rho \int_0^T \int_t^T \EE \Bigg[ \phi''(\Flowx) \p_{x} \sigma (s,\vol_s^{t,\btheta^t}) \int_s^T \p_{xx}^2  \sigma (r,\vol_r^{t,\btheta^t}) K(r-s) \Sigma K(r-t) dB_r \Bigg] \Delta K(s-t) ds dt \\
  & + \int_0^T \int_t^T \EE \big[ \phi''(\Flowx) (\p_{x} \sigma (s,\vol_s^{t,\btheta^t}))^2 \big]  \Delta (K^2)(s-t) ds dt. 
\end{align*}

Once again, we may treat separately the cases $\rho > 0$ and $\rho = 0$ to optimise the constant in the latter. When $\rho > 0$, we apply Hölder, BDG and Minkoswki's integral inequalities to the first two terms. For the third term, we use the Cauchy-Schwarz inequality. Then, we make use of \ref{hyp:regularity} and Lemma \ref{lm:controle_des_moments} as well as \eqref{eq:estimate_mall_phi} to bound the Malliavin derivatives and the moments of the stochastic terms. With this we obtain that there exists $C_{ {H}_1}' >0$, not depending on $K$, such that
\begin{align*}
| {H}_1| 
\leqs & (1+ e^{C_{ {H}_1}' \|K\|_{L^2([0,T])}^2 }) (1+ \rho \|K\|_{L^2([0,T])}) \left( \int_t^T (\Sigma K (r-t))^2 dr \right) ^\frac{1}{2}
  \int_0^T \int_t^T | \Delta K| (s-t) ds dt \\ 
& + \rho  (1+ e^{C_{ {H}_1}' \|K\|_{L^2([0,T])}^2 }) \int_0^T \int_t^T  | \Delta K| (s-t) \bigg( \int_s^T K^2(r-s) (\Sigma K(r-t))^2 dr  \bigg)^{\frac{1}{2}} ds dt  \\
& + (1+ e^{C_{ {H}_1} \|K\|_{L^2([0,T])}^2 }) \int_0^T \int_t^T  | \Delta (K^2)| (s-t) ds dt .
\end{align*}
It is clear that $\left( \int_t^T (\Sigma K (r-t))^2 dr \right) ^\frac{1}{2} \leqs \|K\|_{L^2([0,T])}$.
To bound the integral $\int_s^T K^2(r-s) (\Sigma K(r-t))^2 dr$, we use \ref{hyp:kernel}:
\begin{align*}
\int_s^T K^2(r-s) (\Sigma K(r-t))^2 dr \leq & (1+c_{K,\ak}) \int_s^T K^2(r-s) K^2(r-t) dr \\
  &\leq (1+c_{K,\ak}) K^2(s-t) \int_s^T K^2(r-s) d r \leqs K^2(s-t) \|K\|_{L^2([0,T])}^2.
\end{align*}
In particular, it follows that 
\begin{equation*}
\left( \int_s^T K^2(r-s) (\Sigma K(r-t))^2 dr \right)^{\frac{1}{2}} 
\leqs \|K\|_{L^2([0,T])} \Sigma K(s-t), 
\end{equation*}
hence 
\begin{equation*}
|\Delta K| (s-t) \left( \int_s^T K^2(r-s) (\Sigma K(r-t))^2 dr \right)^{\frac{1}{2}} \leqs \|K\|_{L^2([0,T])} |\Delta (K^2)| (s-t). 
\end{equation*}
This leads, when $\rho > 0$, to the estimate
\begin{align*}
  | {H}_1| \leqs & (1+ e^{C_{ {H}_1}' \|K\|_{L^2([0,T])}^2 }) (\|K\|_{L^2([0,T])}^2 + \rho \|K\|^2_{L^2([0,T])})
  \int_0^T \|\Delta K\|_{L^1([0,t])} dt \\ 
& + (1+\rho)  (1+ e^{C_{ {H}_1}' \|K\|_{L^2([0,T])}^2 }) \int_0^T \int_t^T | \Delta (K^2)| (s-t) ds dt.
\end{align*}

When $\rho = 0$, from \eqref{eq:malliavin_derivative_phi}, we have $\mall{s}{\phi''(\Flowx)} = \phi'''(\Flowx) \sigma(s,\Flowv)$ for all $s \in [t,T]$, and then 
\begin{align*}
   {H}_1 = & \int_0^T \int_t^T \EE \Bigg[ \phi'''(\Flowx) \sigma(s,\Flowv) \p_{x} \sigma (s,\vol_s^{t,\btheta^t}) \int_t^T \p_{x} \sigma (r,\vol_r^{t,\btheta^t}) \Sigma K(r-t) dB_r \Bigg] \Delta K(s-t) ds  dt \\
  & + \int_0^T \int_t^T \EE \big[ \phi''(\Flowx) (\p_{x} \sigma (s,\vol_s^{t,\btheta^t}))^2 \big]  \Delta (K^2)(s-t) ds dt.
\end{align*}
Using that $\phi \in C^4(\RR)$\footnote{This is the only place in the proof where we need $\phi$ to be actually smoother than three time differentiable.} and that $\mall{r}{\sigma(s,\Flowv)} = \mall{r}{\p_{x} \sigma (s,\Flowv)} = 0$ for any $r \in [t,T]$, we can apply the product rule and the Malliavin integration by part to get
\begin{align*}
   {H}_1 = & \int_0^T \!\!\! \int_t^T \!\!\! \int_t^T \!\! \EE \Bigg[ \phi''''(\Flowx) \sigma(s,\Flowv) \p_{x} \sigma (s,\vol_s^{t,\btheta^t}) \sigma(r,\vol_r^{t,\btheta^t}) \p_{x} \sigma (r,\vol_r^{t,\btheta^t}) \Sigma K(r-t) dr  \Bigg] \Delta K(s-t) ds  dt \\
  & + \int_0^T \int_t^T \EE \big[  \phi''(\Flowx) (\p_{x} \sigma (s,\vol_s^{t,\btheta^t}))^2  \big]  \Delta (K^2)(s-t) ds dt.
\end{align*}
Using a $5$-terms Hölder inequality for the first expression in the RHS above and a Cauchy-Schwarz inequality for the second one, we may bound the expectation of the random terms using \ref{hyp:regularity} associated with Lemma \ref{lm:controle_des_moments} and leverage the estimate \eqref{eq:control_sigma_K_L1} to derive that when $\rho = 0$, there exists $C_{ {H}_1}'' >0$ not depending on $K$ such that
\begin{align*}
  | {H}_1| \leqs & (1+ e^{C_{ {H}_1}'' \|K\|_{L^2([0,T])}^2 }) \int_0^T \int_t^T \Big( | \Delta K| (s-t)  + | \Delta (K^2)| (s-t) \Big) ds dt.
\end{align*}
Letting $C_{ {H}_1} = \max({C_{ {H}_1}',C_{ {H}_1}''})$, we deduce that for every $\rho \in [0,1]$:
\begin{align*}
  | {H}_1| \leqs & (1+ e^{C_{ {H}_1} \|K\|_{L^2([0,T])}^2 }) ( 1 + \indic{\rho > 0} (\|K\|_{L^2([0,T])} + \|K\|_{L^2([0,T])}^2) ) \int_0^T \int_t^T | \Delta K | (s-t)  ds dt \\
  & + (1+ e^{C_{ {H}_1} \|K\|^2_{L^2([0,T])} }) \int_0^T \int_t^T | \Delta (K^2)| (s-t)  ds dt.
\end{align*}

\subsubsection*{Estimate of $ {H}_2$ and $ {H}_3$}
The estimates of $ {H}_2 = \int_0^T \int_t^T \EE \Big[ \phi'(\Flowx) b''(s,\Flowv) \Big] \Delta (K^2)(s-t)  ds dt$ \\ and $ {H}_3 = \int_0^T \int_t^T \EE \Big[ \phi''(\Flowx) (\sigma \sigma'')(s,\Flowv) \Delta (K^2)(s-t) ds \Big] dt$ are straightforward, so we treat them together. We apply Cauchy-Schwarz inequality, make use of \ref{hyp:regularity} and Lemma \ref{lm:controle_des_moments}, and obtain the existence of constants $C_{ {H}_2} >0$ and $C_{ {H}_3} >0$, not depending on $K$, such that
\begin{align*}
  | {H}_2| \leqs (1+ e^{C_{ {H}_2} \|K\|_{L^2([0,T])}^2 }) & \int_0^T \int_t^T | \Delta (K^2)| (s-t) ds dt, \\
  | {H}_3| \leqs (1+ e^{C_{ {H}_3} \|K\|_{L^2([0,T])}^2 }) & \int_0^T \int_t^T | \Delta (K^2)| (s-t) ds dt.
\end{align*}

\subsubsection*{Estimate of $ {H}_4$ and $ {H}_5$}
The analysis of ${H}_4 = \int_0^T \int_t^T \EE [ \phi''(\Flowx) \p_{xx} \sigma(s,\Flowv) ( \int_s^T \p_x b(\ell,V_{\ell}^{t,\btheta^t}) K(\ell-s) d\ell) \Delta (K^2)(s-t) ds] dt$ and ${H}_5 = \int_0^T \int_t^T \EE [ \phi''(\Flowx) \newline \p_{xx}\sigma(s,\Flowv) ( \int_s^T \p_x \sigma(\ell,V_{\ell}^{t,\btheta^t}) K(\ell-s) dB_{\ell}) \Delta (K^2)(s-t) ds ] dt$ also just consists in finding a direct upper-bound for the expectation of the stochastic terms, although nothing can be done to remove the square in the difference $\Delta (K^2)$ of the two kernels.

We apply Hölder inequality, Minkoswki's integral inequality for $ {H}_4$ and BDG inequality for $ {H}_5$, use \ref{hyp:regularity} as well as Lemma \ref{lm:controle_des_moments} and the estimates $\int_s^T K(\ell-s) d\ell \leq \|K\|_{L^1([0,T])} \leqs \frac{T^{H+\frac{1}{2}}}{H+\frac{1}{2}} \leq T\vee 1$ and $\int_s^T K^2(\ell-s) d\ell \leq \|K\|_{L^2([0,T])}^2$ to derive:
\begin{align*}
| {H}_4| 
\leq & \int_0^T \int_t^T \mathbb{E} [\phi'' (X_T^{t,\overline{X}_T,\overline{\Theta}^t})^4]^\frac{1}{4} 
\EE[\p_{xx} \sigma (s, V_s^{t,  \overline{\Theta}^t})^2]^{\frac{1}{2}} 
\left( \int_s^T \EE[ \p_{x} b(\ell,V_{\ell}^{t,\btheta^t})^4]^\frac{1}{4} K(\ell-s) d\ell \right) 
|\Delta (K^2)| (s-t) ds dt \\
& \leqs (1 + e^{C_{ {H}_4} \| K \|_{L^2([0,T])}^2 })\int_0^T\int_t^T |\Delta (K^2)| (s-t) ds dt,
\end{align*}
and 
\small{
\begin{align*}  
| {H}_5| 
\leq & \ c_4^{\frac{1}{4}} \int_0^T \int_t^T \mathbb{E} [\phi'' (X_T^{t,\overline{X}_T,\overline{\Theta}^t})^4]^\frac{1}{4} 
\EE[ \p_{xx} \sigma (s, V_s^{t,\overline{\Theta}^t})^2]^{\frac{1}{2}}  
\left( \int_s^T \mathbb{E} [ \p_x \sigma(\ell,V_{\ell}^{t, \overline{\Theta}^t})^4]^{\frac{1}{2}} K^2 (\ell - s) d \ell \right)^{\frac{1}{2}}|    \Delta (K^2)|  (s - t)  d s d t \\
& \leqs \|K\|_{L^2([0,T])} (1 + e^{C_{ {H}_5} \| K \|_{L^2([0,T])}^2 }) \int_0^T \int_t^T |
\Delta (K^2) | (s - t)  d s d t,
\end{align*}
}
for some $C_{ {H}_4},C_{ {H}_5} > 0$ not depending on $K$.

\subsubsection*{Final upper-bound for the error}

Revisiting the decomposition \eqref{eq:error_decompos} of the error $\Ee_T$, and incorporating the previous estimations, we conclude that, for every $\rho \in [0,1]$,
\begin{equation*}
|\mathcal{E}_T| \leqs (1 + e^{C \|K\|_{L^2([0,T])}^2}) (1 + \indic{\rho > 0} (\|K\|_{L^2([0,T])} +  \|K\|_{L^2([0,T])}^2)) \int_0^T \big(\|\Delta K \|_{L^1([0,T])} + \| \Delta (K^2) \|_{L^1([0,T])}\big) dt,
\end{equation*}
where $\displaystyle{C \coloneq \max_{1\leq i\leq 5,\ 1\leq j\leq 5}(C_{ {F}_i}, C_{ {H}_j})}$, which ends the proof of Theorem \ref{thm:weakcvg_integrated}.

\subsection*{Acknowledgements}
This work benefited from stimulating discussions with J\'{e}r\'{e}mie Bec and Simon Thalabard who are warmly acknowledged. \newline M.B. and P.M. acknowledge the support of the French Agence Nationale de la Recherche (ANR), under grant ANR-21-CE30-0040-01 (NETFLEX). K.M. acknowledges the support of ANID FONDECYT/POSTDOCTORADO N${}^\circ$ 3210111 and the Centro de Modelamiento Matemático (CMM) BASAL fund FB210005 for Center of Excellence from ANID-Chile.

\appendix
\section{Appendix} \label{sec:appendix}
\subsection{Hadamard derivatives in $L^p(\Omega)$, chain rule and product rule}\label{sec:appendix_hadam}

To introduce the notion of differentiability in Hadamard's sense and the associated chain rule, we adopt the formalism presented in \citet{shapiro1990}.

\begin{definition}
Let $(X,d_X)$ and $(Y,d_Y)$ be two metric spaces and $\func{f}{X}{Y}$.  $f$ is called Hadamard-differentiable at $x \in X$ in the direction $h \in X$ if,  for any sequence $(h_n)_{n \in \NN}$ of elements of $X$ and $(t_n)_{n \in \NN}$ of elements of $\RR_+$ such that $h_n \xrightarrow{} h$ and $t_n \xrightarrow{} 0$, the limit of the sequence
\begin{equation*}
\left(\frac{f(x + t_n h_n) - f(x) }{t_n}\right)_{n \in \NN}
\end{equation*}
exists in $Y$. We denote this limit by $Df(x,h)$ and call it the Hadamard derivative of $f$ at $x$ in the direction $h$.
\end{definition}

When $f$ is Hadamard-differentiable at $x$ in the direction $h$, then $f$ is also G\^{a}teaux-differentiable at $x$ in the direction $h$ and the G\^{a}teaux derivative is equal to the Hadamard derivative. The chain rule for Hadamard derivatives is given by the following proposition (the proof can be found in \citet{shapiro1990}, Proposition 3.6).
\begin{proposition} \label{prop:chain_rule}
Let $X,Y,Z$ be metric spaces, $\func{f}{X}{Y}$ and $\func{g}{Y}{Z}$. If $f$ is G\^{a}teaux-differentiable at $x \in X$ in the direction $h \in X$ and $g$ is Hadamard-differentiable at $f(x) \in Y$ in the direction $Df(x,h) \in Y$, then $g \circ f$ is G\^{a}teaux-differentiable at $x$ in the direction $h$ and 
\begin{align*}
D (g \circ f)(x,h) = Dg(f(x),Df(x,h)).
\end{align*}
\end{proposition}

The following lemma gives a sufficient condition for the Hadamard differentiability of the extension of a smooth function $\func{f}{\RR}{\RR}$ to the appropriate probabilistic spaces in our specific context of application.
\begin{lemma} \label{lm:hadamard_differentiability}
Let $f \in C^2(\RR,\RR)$, and assume that, for $g \in \{ f, f', f'' \}$, there exists a real $p \geq 1$ such that $|g(x)| \leqs 1 + |x|^p$ for any $x \in \RR$. Fix $r \geq 1$. For $q\geq 2 p r \vee 4 r$, we consider the application $\func{F}{L^q(\Omega)}{L^r(\Omega)}$ defined by $F(X)(\omega) = f(X(\omega))$ for $X \in L^q(\Omega)$ and $\omega \in \Omega$. 
Then $F$ is a well-defined Hadamard-differentiable application between the Banach spaces $(L^q(\Omega),\| \cdot\|_{q})$ and $(L^r(\Omega),\|\cdot\|_{r})$, and the Hadamard derivative of $F$ at $X \in L^q(\Omega)$ in the direction $H \in L^q(\Omega)$ is given by $F'(X,H) = f'(X) H$. 
\end{lemma}

\begin{proof}
Firstly, we note that the application $F$ is well-defined since for $X \in L^q(\Omega)$, we have $\EE[|F(X)|^r] = \EE[|f(X)|^r] \leqs 1 + \EE[|X|^{pr}] < +\infty$, as $p r \leq q$. To show the Hadamard-differentiability of $F$ at $X \in L^q(\Omega)$ in the direction $H \in L^q(\Omega)$, we consider a sequence $(H_n)_{n \in \NN}$ of elements of $L^q(\Omega)$ and $(t_n)_{n \in \NN}$ of elements of $\RR_+$ such that $H_n \xrightarrow{} H$ in $L^q$ and $t_n \xrightarrow{} 0$. We  employ a second order Taylor expansion of $f$ to obtain
\begin{equation*}
\frac{f(X + t_n H_n) - f(X)}{t_n} -  f'(X) H  = f'(X)(H_n - H) + \frac{1}{t_n} (t_n H_n)^2 \int_0^1 (1-\lambda) f''(X + \lambda t_n H_n) \  d \lambda.
\end{equation*}
Introducing $f'(X) H_n$ as a pivot term, by the Cauchy-Schwarz inequality and Minkowski's integral inequality, we deduce that
\begin{equation*}
\Bigg\| \frac{f(X + t_n H_n) - f(X)}{t_n} -  f'(X) H_n \Bigg\|_r \leq  t_n \| H_n \|_{4r}^2 \int_0^1 (1-\lambda) \| f''(X + \lambda t_n H_n) \|_{2r} \ d \lambda.
\end{equation*}
The polynomial growth hypothesis on $f''$ then leads to 
\begin{equation*}
\| f''(X + \lambda t_n H_n) \|_{2r} \leq \| 1 + |X + \lambda t_n H_n|^p \|_{2r}  \leq 1 + \| X + \lambda t_n H_n \|_{2 p r}^p  \leq 1 + (\| X \|_{2pr} + \lambda t_n \| H_n \|_{2pr})^p.
\end{equation*}
Since $(H_n)_{n \in \NN}$ converges to $H$ in $L^q(\Omega)$ with $q\geq 2pr \vee 4r$, it also converges in $L^{2pr}(\Omega)$, and the polynomial growth hypothesis on $f'$ ensures that $f'(X) \in L^r(\Omega)$ and then that $\|f'(X)(H-H_n)\|_r$ vanishes.  
Moreover the sequence $(H_n)_{n \in \NN}$ is bounded in $L^{2pr}(\Omega)$ and in $L^{4r}(\Omega)$, and we have $X$ in $L^{2pr}(\Omega)$, allowing to conclude that the sequence $(f(X + t_n H_n) - f(X))/t_n$ converges in $L^r(\Omega)$ to $f'(X) H$. This shows that $f$ is Hadamard-differentiable at $X$ in the direction $H$, and the Hadamard derivative is equal to $f'(X) H$.
\end{proof}

The product rule for two $L^2(\Omega)$-valued functions of a metric space is formalized as follows. 
\begin{lemma} \label{lm:product_rule}
Let (X,d) be a metric space. If two functions $\func{f, g}{X}{L^2(\Omega)}$ are G\^{a}teaux-differentiable at $x \in X$ in the direction $h \in X$, then their product $\func{fg}{X}{L^1(\Omega)}$ is G\^{a}teaux-differentiable for the norm $\| \cdot \|_1$ at $x$ in the direction $h$ and the G\^{a}teaux derivative is given by $D(f g)(x,h) = f(x)Dg(x,h) + g(x)Df(x,h)$.
\end{lemma}
\begin{proof}
Let $x,h \in X$ and $\varepsilon > 0$. Writing 
\begin{equation*}
\frac{f(x + \varepsilon h) g(x + \varepsilon h) - f(x)g(x)}{\varepsilon} = f(x) \frac{g(x + \varepsilon h) - g(x)}{\varepsilon} + g(x + \varepsilon h) \frac{f(x + \varepsilon h) - f(x)}{\varepsilon}, 
\end{equation*}
we conclude by the Cauchy-Schwarz inequality and letting $\varepsilon$ goes to zero.
\end{proof}

\subsection{Proof of Lemma \ref{lm:controle_des_moments} (on moment bounds)} \label{sec:proof_controle_des_moments}

We only prove the last two inequalities in Lemma \ref{lm:controle_des_moments} since the two first ones are obtained with similar arguments: we start by proving exponential moments on $V$. Let $p \geq 1$, $t \in [0,T]$ and $s \in [t,T]$. The random variable $\vol_s^{t,\btheta^t}$ follows a centred Gaussian distribution with variance $\int_0^t \ak^2(s-r) dr + \int_t^s K^2(s-r) dr$. Then 
\begin{equation*}
\EE[\exp(p \vol_s^{t,\btheta^t})] 
= \exp\left( \frac{p^2}{2} \left( \int_0^t \ak^2(s-r) dr + \int_t^s K^2(s-r) dr \right) \right). 
\end{equation*}
From \ref{hyp:kernel}, $\ak^2(s-r) \leq c^2_{K,\ak} K^2(s-r)$ for any $r \leq s$, and  hence
\begin{equation} \label{eq:ineq_moments_exp_v}
\begin{aligned}
\EE[\exp(p \vol_s^{t,\btheta^t})] 
\leq \ \exp\big(\frac{p^2}{2} (1+c_{K,\ak}^2) \int_0^s K^2(s-r) dr\big)  
\leq \ \exp\big( \frac{p^2}{2} (1+c_{K,\ak}^2) \| K \|_{L^2([0,T])}^2\big).
\end{aligned}
\end{equation}
Setting $\overline{m}^{(2)}_p = \frac{p^2}{2} (1+c_{K,\ak}^2)$, we obtain the desired result.
We consider now the $p$-moment of $X$. Using the standard inequality $|x+y+z|^p \leq 3^{p-1} (|x|^p + |y|^p + |z|^p)$, we have
\begin{equation*}
    \EE[ | \itgr_T^{t,\bitgr_t,\btheta^t} |^p] \leq 3^{p-1} \left\{ \EE[ | \bitgr_t |^p] + \EE\left[ \left| \int_t^T b(s,\vol_s^{t,\btheta^t}) ds \right|^p \right] + \EE\left[ \left| \int_t^T \sigma(s,\vol_s^{t,\btheta^t}) dB_s \right|^p \right] \right\}.
\end{equation*}
We apply the Jensen and Burkholder-Davis-Gundy inequalities to the last two terms and make use of the Fubini-Tonelli theorem to push the expectations inside the Lebesgue integrals. We obtain:
\begin{equation*}
    \EE[ | \itgr_T^{t,\bitgr_t,\btheta^t} |^p] \leq 3^{p-1} \left\{ \EE[ | \bitgr_t |^p] + T^{p-1} \int_t^T \EE[ | b(s,\vol_s^{t,\btheta^t}) |^p] ds + c_{p} T^{p/2-1} \int_t^T \EE[ | \sigma(s,\vol_s^{t,\btheta^t}) |^p] ds  \right\},
\end{equation*}
where $c_{p}$ is the BDG constant. Form  \ref{hyp:regularity},  $|b(s,x)| \leqs 1 + \kappa_{b} e^{\nu_b (x+s)}$ and $|\sigma(s,x)| \leqs 1 + \kappa_{\sigma} e^{\nu_{\sigma} (x+s)}$. It follows, using \eqref{eq:ineq_moments_exp_v}, that 
\begin{align*}
 \EE[ | b(s,\vol_s^{t,\btheta^t}) |^p] \leq 2^{p-1} \left( 1 + \kappa_{b}^p \EE[e^{ p \nu_b (\vol_s^{t,\btheta^t} + s) } ] \right) \leqs 2^{p-1} (1 + \kappa_{b}^p e^{ p \nu_b T } e^{\overline{m}_{p \nu_b}^{(2)} \| K\|_{L^2([0,T])}^2 } ).
\end{align*}
Moreover, the same argument applies to the term with $\sigma$:
\begin{align*}
  \EE[ | \sigma(s,\vol_s^{t,\btheta^t}) |^p] \leq 2^{p-1} (1 + \kappa_{\sigma}^p e^{ p \nu_{\sigma} T } e^{\overline{m}_{p \nu_{\sigma}}^{(2)} \| K\|_{L^2([0,T])}^2 } ).
\end{align*}
We use the same strategy to obtain a bound for $\EE[ | \bitgr_t |^p]$:
\begin{align*}
  \EE[ | \bitgr_t |^p] \leq & 3^{p-1} \EE \left\{ |\itgr_0|^p +  \left| \int_0^t b(s,\bvol_s) ds \right|^p + \left| \int_0^t \sigma(s,\bvol_s) dB_s \right|^p \right\} \\
  & \leq 3^{p-1} \left\{ \EE[ | \itgr_0 |^p] + T^{p-1} \int_0^T \EE[ | b(s,\bvol_s) |^p] ds + c_{p} T^{p/2-1} \int_0^T \EE[ | \sigma(s,\bvol_s) |^p] ds  \right\}.
\end{align*}
The growth control argument of $b$ and $\sigma$ apply as well, and we can use \eqref{eq:ineq_moments_exp_v} again thanks to the identity $\bvol_s = \vol_s^{s,\btheta^s}$, while $\EE[|\itgr_0|^p] < + \infty$ by hypothesis. We conclude by plugging the estimate of the three terms in the first inequality and setting $m^{(1)}_p = \max(m^{(2)}_{p \nu_b},m^{(2)}_{p \nu_{\sigma}})$. 

\subsection{Proof of Lemma \ref{lm:extended_continuity}} \label{sec:proof_extended_continuity}

Let $(X_n)_{n \in \NN}$ be a sequence of random variable in $L^q(\Omega)$ converging to $X \in L^q(\Omega)$. Using Cauchy-Schwarz inequality and Minkoswki's integral inequality, we have
\begin{align*}
  \EE[ | f(X_n) - f(X) | ] = & \EE \bigg[ |X_n - X| \Big| \int_0^1 f'(\lambda X_n + (1-\lambda) X) d \lambda \Big| \bigg] \\ & \leq \| X_n - X \|_{L^2(\Omega)} \int_0^1 \|  f'(\lambda X_n + (1-\lambda) X) \|_{L^2(\Omega)} d \lambda.
\end{align*}
By hypothesis, $|f'(\lambda X_n + (1-\lambda) X)| \leq 1 + | \lambda X_n + (1-\lambda) X |^p$, so it follows that
\begin{equation*}
  \|  f'(\lambda X_n + (1-\lambda) X) \|_{L^2(\Omega)} \leq (2 + 2 \EE[| \lambda X_n + (1-\lambda) X |^{2p} ])^{\frac{1}{2}} \leq (2 + 2^{2p} (\EE[|X_n|^{2p}] + \EE[|X|^{2p}]) )^{\frac{1}{2}}.
\end{equation*}
Since $X_n \to X$ in $L^q(\Omega)$ with $q \geq 2p$, the sequence $(X_n)_{n \in \NN}$ is bounded in $L^{2p}(\Omega)$ and $X \in L^{2p}(\Omega)$. Moreover, $X_n \to X$ in $L^q(\Omega)$ with $q \geq 2$ so in particular, $X_n \to X$ in $L^2(\Omega)$, leading to the expected result.

%%%%%%%%%%%%%%%%%%%%%%%%%%%%%%%%%%%%%%%%%%%%%%%%%%%%%%%%%%%%%%%%%%%%%%%%%
% Biblio
\bibliographystyle{plainnat}
%\bibliography{biblio_paper.bib}

\end{document}